\documentclass{amsart}

\usepackage{epsfig}
\usepackage{mathtools}
\usepackage{amsmath}
\usepackage{amssymb}
\usepackage{amsfonts}
\usepackage{amscd}
\usepackage{graphicx}
\usepackage{color}
\usepackage{verbatim}

\newtheorem{thm}{Theorem}[section]
\newtheorem{lem}[thm]{Lemma}
\newtheorem{prop}[thm]{Proposition}

\newtheorem{ques}[thm]{Question}

\newtheorem{cor}[thm]{Corollary}
\newtheorem{defn}[thm]{Definition}

\theoremstyle{remark}
\newtheorem{rem}[thm]{Remark}

\newcommand{\inv}{invariant}
\newcommand{\eps}{\varepsilon}
\newcommand{\im}{invariant measure}
\newcommand{\sq}{sequence}

\newcommand{\z}{\mathbb Z}
\newcommand{\na}{\mathbb N}

\newcommand{\M}{\mathcal M}

\newcommand{\A}{\mathcal A}
\newcommand{\B}{\mathcal B}

\newcommand{\tl}{topological}

\newcommand{\zd}{zero-dimensional}

\newcommand{\CC}{\mathcal{C}}

\newcommand{\CT}{\mathcal{T}}

\newcommand{\CS}{\mathcal{S}}

\numberwithin{equation}{section}

\begin{document}

\baselineskip=12pt
\title{The comparison property of amenable groups}

\author{Tomasz Downarowicz}
\address{Faculty of Pure and Applied Mathematics, Wroc\l aw University of Science and Technology, Wybrze\.ze Wyspia\'nskiego 21, 50-370 Wroc\l aw, Poland}
\email{Tomasz.Downarowicz@pwr.edu.pl}


\author{Guohua Zhang}
\address{School of Mathematical Sciences and Shanghai Center for Mathematical Sciences, Fudan University, Shanghai 200433, China}
\email{chiaths.zhang@gmail.com}

\begin{abstract}
Let a countable amenable group $G$ act on a \zd\ compact metric space $X$. For two clopen subsets $\mathsf A$ and $\mathsf B$ of $X$ we say that $\mathsf A$ is \emph{subequivalent} to $\mathsf B$ (we write $\mathsf A\preccurlyeq \mathsf B$), if there exists a finite partition $\mathsf A=\bigcup_{i=1}^k \mathsf A_i$ of $\mathsf A$ into clopen sets and there are elements $g_1,g_2,\dots,g_k$ in $G$ such that $g_1(\mathsf A_1), g_2(\mathsf A_2),\dots, g_k(\mathsf A_k)$ are disjoint subsets of $\mathsf B$. We say that the action \emph{admits comparison} if for any clopen sets $\mathsf A, \mathsf B$, the condition, that for every $G$-\inv\ probability measure $\mu$ on $X$ we have the sharp inequality $\mu(\mathsf A)<\mu(\mathsf B)$, implies $\mathsf A\preccurlyeq \mathsf B$. Comparison has many desired consequences for the action, such as the existence of tilings with arbitrarily good F\o lner properties, which are factors of the action. Also, the theory of symbolic extensions, known for $\z$-actions, extends to actions which admit comparison. We also study a purely group-theoretic notion of comparison: if every action of $G$ on any \zd\ compact metric space admits comparison then we say that $G$ has the \emph{comparison property}. Classical groups $\z$ and $\z^d$ enjoy the comparison property, but in the general case the problem remains open. In this paper we prove this property for groups whose every finitely generated subgroup has subexponential growth.
\end{abstract}

\thanks{The first author is supported by the NCN (National Science Center, Poland) Grant 2013/08/A/ST1/00275.
The second author is supported by NSFC (National Natural Science Foundation of China) Grants
11671094, 11722103 and 11731003.}

\date{\today}

 \maketitle
\section{Introduction}
The key notion of this paper, the \emph{comparison} originates in the theory of \mbox{$C^*$-al}geb\-ras, but the most important for us ``dynamical'' version concerns group actions on compact spaces. In this setup it was defined by J.~Cuntz (see \cite{Cu}) and further investigated by M.~R\o rdam in \cite{MR1,MR2} and by W.~Winter in \cite{W}. As in the case of many other properties and notions in dynamical systems, the most fundamental form of comparison occurs in actions of the additive group $\z$ of the integers. In this context comparison is guaranteed for any action on a \zd\ compact metric space, which follows from the classical marker property of such actions (see \cite{Bo}). See also \cite{B} for more on comparison in $\z$-actions.
For a wider generality, we refer the reader to a recent paper by David Kerr \cite{K}, where the notion is defined for other actions including \tl\ and measure-preserving ones. We will focus on a particular case where a countable amenable group acts on a \zd\ compact metric space. In fact, this case also plays one of the leading roles in \cite{K}.

The main motivation for this paper is the fact that, unlike for $\z$-actions, in the case of a general countable amenable group acting on a \zd\ compact metric space, it is unknown whether comparison necessarily occurs. There is neither a proof, nor a counterexample, although the problem has been attacked by several specialists for several years. Only a few partial results have been obtained, for instance, it is known (but never published, see \cite{Ph} and also \cite{SG}) that finitely generated groups with a symmetric F\o lner \sq\ satisfying Tempelman's condition (this includes all nilpotent, in particular Abelian, groups) have the comparison property, but beyond this case not much was known. On the other hand, comparison is a very desirable property with many important consequences (see further in this introduction), thus any progress in understanding which actions enjoy comparison (or which groups have comparison for all actions) is valuable.

Our main invention introduced in this paper is a new notion of a \emph{correction chain}---a kind of pseudoorbit which allows to improve a partially defined map and extend its domain.
Using this tool in section \ref{cztery} we succeed in identifying a large class of groups whose any action on a \zd\ compact metric space has comparison. Namely, it is the class of groups whose every finitely generated subgroup has subexponential growth (we call them shortly \emph{subexponential groups}). This covers all nilpotent and in fact virtually nilpotent groups (which have polynomial growth) but also other, with intermediate growth, the most known example of which is the Grigorchuk group (\cite{G}). By a recent result of Breuillard, Green and Tao \cite{BGT}, our result also covers the above mentioned ``Tempelman groups''; they turn out to be virtually nilpotent.
Of course, there exist also amenable groups with exponential growth, and for these the problem remains a challenge.
\smallskip

The last section of the paper is devoted to the connection between comparison and the existence of what we call \emph{dynamical tilings} with arbitrarily good ``F\o lner properties''. Such dynamical tilings, which exist in any aperiodic action of $\z$, have numerous applications in ergodic theory and occur under various names (as Kakutani--Rokhlin partitions or clopen tower partitions, etc.) for example in the study of full groups and orbit equivalence of minimal Cantor systems (see \cite{Sl} for an exposition on this subject). For amenable group actions, for a long time, quasitilings of Ornstein and Weiss (see \cite{OW1}) have played a crucial role, mainly due to their universal existence in all countable amenable groups, and Lindenstrauss' Pointwise Ergodic Theorem (\cite{L}) is one of the most important applications. There are many more such applications, see for example \cite{DZ,FT,HYZ,PS}. However, the Ornstein--Weiss quasitilings are ``algebraic'' (i.e., unrelated to any \emph{a priori} given action).

In \cite{DHZ}, it has been proved that in the algebraic case the Ornstein--Weiss quasitilings (with arbitrarily good F\o lner properties) can be improved to become tilings. Such tilings have already found numerous applications, see e.g. \cite{D,S,Z,ZCY}. In fact, the advantage of (algebraic) tilings over (algebraic) quasitilings is visible also in this paper: these tilings are used in the proof of the key Lemma \ref{key}, where quasitilings would not work.

As mentioned above, it is often desired to have a tiling (or at least a quasitiling) which depends on the \emph{a priori} given action. In \cite{DH} it is proved that dynamical quasitilings with arbitrarily good F\o lner properties exist \emph{as factors} in any free action of any countable amenable group. In this version the result has already been used in \cite{FH, FH1}. In a recent paper \cite{C-T} we find a different approach: a dynamical tiling (in place of quasitiling) is obtained as \emph{a factor of an extension} of a given free minimal action. In the same paper, these tilings are applied to establish some kind of stability for generic actions, using the language of $C^*$-algebras (see also the survey \cite{W1}).

But for many other purposes all the above discussed quasitilings and tilings are insufficient. Dynamical tilings which are factors of a given action are needed for instance to build the theory of symbolic extensions, and neither dynamical quasitilings nor tilings which are factors of some \emph{extended} action seem to be sufficient. This problem will be discussed in detail in our forthcoming paper \cite{DoZ}.

As we have already mentioned, in the general case the existence of a dynamical tiling which is a factor of a given free action (on a \zd\ compact metric space) is unknown. In the last section, we tie the existence of such dynamical tilings with the comparison property. In particular, using the result concerning the comparison property of subexponential groups, we prove that if the group is subexponential, then any free action on a \zd\ compact metric space factors to dynamical tilings with arbitrarily good F\o lner properties. We also prove the reversed implication: an action (not necessarily free), which factors to dynamical tilings with arbitrarily good F\o lner properties, admits comparison.

\medskip
The authors thank Gabor Szabo for valuable information on the current state of the art in the subject matter.
\section{Preliminaries}
In this paper, whenever we say ``a finite set'' we mean a nonempty finite set, and whenever we say ``a countable set'' we mean either a finite or an infinite countable set. Since we will often consider pairs of subsets of a group $G$ as well as pairs of subsets of a compact space $X$, for easier distinction we will use the convention that in the first case these sets will be denoted using slanted font: $A,B\subset G$, and in the second using straight sans serif font: $\mathsf A,\mathsf B\subset X$.
Boldface letters $\mathbf A,\mathbf B$ are reserved for blocks in symbolic systems, while script $\A,\B$ will be used for families of blocks.

\subsection{Amenable groups and their actions}

Let $G$ be a countable group.
\begin{defn}
A \sq\ $(F_n)_{n\in\na}$ of finite subsets of $G$ is called a (left) \emph{F\o lner \sq} if for any $g\in G$ one has
\[
\lim_{n\to\infty}\frac{|gF_n\cap F_n|}{|F_n|}=1,
\]
where $|\cdot|$ denotes the cardinality of a set.
\end{defn}

Equivalently, a sequence of finite sets $(F_n)$ is F\o lner if and only if for every finite set $K\subset G$ and every $\varepsilon>0$ the sets $F_n$ are eventually \emph{$(K,\varepsilon)$-invariant}, i.e., satisfy
$$
\frac{|KF_n\triangle F_n|}{|F_n|}<\varepsilon
$$
($\triangle$ stands for the symmetric difference of sets).

\begin{defn}
A countable group possessing a F\o lner \sq\ is called \emph{amenable}.
\end{defn}

The above is just one of many equivalent definitions of amenability, applicable to countable groups. For more general definitions and properties see for example \cite{P}.
In particular, it is known that a subgroup of an amenable group is amenable.

\medskip
Let $X$ be a \tl\ space. We say that a group $G$ \emph{acts on $X$} if there is a group homomorphism
$\tau:G\to\mathsf{HOMEO}(X,X)$ of $G$ into the group of self-homeomorphisms of $X$. If an action of $G$ on $X$ is understood, it is customary to write $g(x)$ instead of $\tau(g)(x)$.
By the \emph{orbit} of a point $x\in X$ we will mean the set $G(x)=\{g(x):g\in G\}$.
It is a basic property of amenability (and in fact a condition equivalent to it) that if $G$ is amenable then for any action of $G$ on any compact metric space there exists a Borel probability measure $\mu$ on $X$ invariant under the action in the following sense: if $\mathsf A\subset X$ is a Borel set then $\mu(\mathsf A)=\mu(g(\mathsf A))$ for every $g\in G$. We will briefly call such $\mu$ an \emph{\im} (skipping the adjectives ``Borel'' and ``probability''). The collection of all \im s $\M_G(X)$ endowed with the weak-star topology is a compact convex set whose extreme points are precisely the ergodic measures (i.e., such that each Borel-measurable invariant set has measure either zero or one). If $(F_n)$ is any F\o lner \sq\ in $G$ then, for every point $x\in X$, the sequence of atomic measures
$$
\frac1{|F_n|}\sum_{g\in F_n}\delta_{g(x)}
$$
(where $\delta_x$ denotes the point-mass at $x$) accumulates at the set of \im s. If this \sq\ converges to some $\mu$ then $\mu$ is necessarily invariant and we call $x$ \emph{a generic point for $\mu$}. If $(F_n)$ is \emph{tempered}, then for every ergodic measure $\mu$ the set of all points generic for $\mu$ has full measure (see \cite{L} for the definition of the notion ``tempered'' and for the theorem).
\smallskip

The most basic example of an action is the \emph{full shift on a finite alphabet}. Let $\Lambda$ be a finite set (called \emph{the alphabet}) and let $\Lambda^G$ be endowed with the product topology, where the set $\Lambda$ is considered discrete. The group $G$ acts on $\Lambda^G$ by the \emph{shifts} defined as follows: if $x=(x_g)_{g\in G}$ and $h\in G$ then $h(x)=y=(y_g)_{g\in G}$, where, for each $g\in G$, $y_g = x_{gh}$.

By a \emph{subshift} we will understand the action of $G$ on any nonempty closed shift-\inv\ subset $X$ of $\Lambda^G$. If $F\subset G$ is finite then any function $\mathbf B:F\to\Lambda$ is called a \emph{block} over $F$. With each block $\mathbf B$ (over any finite $F$) we associate the \emph{cylinder}
$$
[\mathbf B]=\{x\in \Lambda^G: x|_F=\mathbf B\}.
$$
If $F=\{e\}$ (where $e$ denotes the unity of $G$) and $\mathbf B(e) = \alpha\in\Lambda$
($\mathbf B(g)$ denotes the entry of $\mathbf B$ at the coordinate $g$) then the cylinder $[\mathbf B]$ will be denoted by $[\alpha]$. We say that a block $\mathbf B$ (over some $F$) \emph{occurrs} in some $x\in \Lambda^G$ \emph{at the position $g$} if $g(x)\in[\mathbf B]$, equivalently, if $x_{fg}=\mathbf B(f)$ for every $f\in F$. If we restrict our attention to a subshift $X\subset \Lambda^G$, by the cylinder $[\mathbf B]$
we will understand what should be formally denoted as $[\mathbf B]\cap X$. The collection of all cylinders (corresponding to all blocks over all finite sets) is a clopen base of the topology in $X$.

If $G$ acts on two compact metric spaces, $X$ and $Y$, we will say that the action on $Y$ is a \emph{\tl\ factor} of the action on $X$ if there exists a continuous surjection $\pi:X\to Y$ such that, for every $g\in G$, $g\circ\pi = \pi\circ g$  (where $g$ is understood as a transformation of either $X$ or $Y$).
\smallskip

The property of a group action on a \tl\ space which generalizes that of aperiodicity for $\z$-actions
(i.e., lack of periodic points) is freeness. There are several (not equivalent) definitions of a free group action. We will use the strongest:
\begin{defn}
An action $\tau:G\to\mathsf{HOMEO}(X,X)$ is called \emph{free} if for every $g\in G$
$$
(\exists\,x\in X : g(x)=x) \text{ \ implies \ }g=e.
$$
\end{defn}

\subsection{Subexponential groups}

\begin{defn}
In a group $G$, a set $R$ such that $\bigcup_{n=1}^\infty (R\cup R^{-1})^n=G$ is called a \emph{generator} of $G$. A group having a finite generator is called \emph{finitely generated}.
\end{defn}

\begin{defn}
A finitely generated group $G$ with a generator $R$ has \emph{subexponential growth} if $|(R\cup R^{-1})^n|$ grows subexponentially, i.e.,
$$
\lim_{n\to\infty}\frac1n\log|(R\cup R^{-1})^n|=0.
$$
\end{defn}

It is very easy to see that subexponential growth of a finitely generated group $G$ implies subexponential growth of $|K^n|$ for any finite set $K\subset G$ and thus does not depend on the choice of a finite generator.

\begin{defn}
A countable group $G$ (not necessarily finitely generated) is called \emph{subexponential} if every its finitely generated subgroup has subexponential growth.
\end{defn}

It is a standard fact that a group $G$ is amenable if and only if so is every finitely generated subgroup of $G$. It is also known that finitely generated groups with subexponential growth are amenable \cite{AS}, hence every subexponential group is amenable. This is why we can omit amenability assumption when dealing with subexponential groups. Examples of subexponential groups are: Abelian, nilpotent and virtually nilpotent groups. These examples have polynomial growth, but there are also examples of countable groups with intermediate growth rates (see \cite{G}). By a recent result \cite{BGT}, all finitely generated groups, which admit an increasing \sq\ of sets $(A_n)_{n\in\na}$ with $G=\bigcup_{n=1}^\infty A_n$ and $|A_n^2|<C|A_n|$ for some constant $C>0$, are virtually nilpotent and hence subexponential. In particular, this applies to finitely generated groups possessing a symmetric F\o lner \sq\ $(F_n)$ satisfying Tempelman's condition $|F_n^{-1}F_n|\le C|F_n|$.

\subsection{Upper and lower Banach densities, Banach density advantage}

\begin{defn}\label{1.7}
For a subset $B\subset G$ and a finite set $F\subset G$ denote
\[
\underline D_F(B)=\inf_{g\in G} \frac{|B\cap Fg|}{|F|}\text{ \ and \ } \overline D_F(B)=\sup_{g\in G} \frac{|B\cap Fg|}{|F|}.
\]
If $(F_n)$ is a F\o lner sequence then define
\[
\underline D(B)=\limsup_{n\to\infty} \underline D_{F_n}(B)\text{ \ and \ }
\overline D(B)=\liminf_{n\to\infty} \overline D_{F_n}(B),
\]
which we call the \emph{lower} and \emph{upper Banach density} of $B$, respectively.

\begin{rem}
The notions of upper and lower Banach density have been studied from several points of view.
For example, in \cite{BBF} the reader will find a different definition. It can be shown that that definition is in fact equivalent to ours.
\end{rem}

For two sets $A$ and $B$ of $G$ we define the following quantities
\[
\underline D_F(B,A)=\inf_{g\in G} \frac1{|F|}(|B\cap Fg|-|A\cap Fg|), \ \ \
\underline D(B,A)=\limsup_{n\to\infty} \underline D_{F_n}(B,A).
\]
The latter number will be called the \emph{Banach density advantage} of $B$ over $A$ (which can be negative, but we will never consider such a case).
\end{defn}

We will be using the following elementary fact (see e.g. \cite[Lemma 3.4]{DHZ}, where it is formulated using the language of quasitilings which in this paper will be introduced later).

\begin{lem}\label{1.5}
Let $(A_k)_{k\ge 1}$ and $(g_k)_{k\ge 1}$ be a sequence of subsets of $G$ and a \sq\ of elements of $G$ such that:
\begin{enumerate}
	\item the union $\bigcup_{k=1}^\infty A_k$ is finite,
	\item $A=\bigcup_{k=1}^\infty A_kg_k$ is a disjoint union.
\end{enumerate}
For each $k$ let $B_k\subset A_k$ and let $B=\bigcup_{k=1}^\infty B_kg_k$. Then
$$
\underline D(B)\ge\underline D(A)\cdot\inf_k\frac{|B_k|}{|A_k|}.
$$
\end{lem}

The following lemma will be repeatedly used in many of our considerations.
\begin{lem}\label{bdc}
Let $F, F_1$ be finite subsets of $G$ and let $A,B$ be some arbitrary subsets of $G$.
If $F_1$ is $(F,\eps)$-\inv\ then $\underline D_{F_1}(B,A)\ge \underline D_F(B,A)-4\eps$.
\end{lem}

\begin{proof}
Given $g\in G$, we have
$$
|B\cap Fhg|-|A\cap Fhg|\ge \underline D_F(B,A)|F|,
$$
for every $h\in F_1$. This implies that
\begin{multline*}
|\{(f,h): f\in F, h\in F_1, fhg\in B\}| - |\{(f,h): f\in F, h\in F_1, fhg\in A\}| \ge\\
\underline D_F(B,A)|F||F_1|.
\end{multline*}
This in turn implies that there exists at least one $f\in F$ for which
$$
|B\cap fF_1g|-|A\cap fF_1g|\ge \underline D_F(B,A)|F_1|.
$$
Since $f\in F$ and $F_1$ is $(F,\eps)$-invariant (and hence so is $F_1g$), we have
$$
\bigl||B\cap fF_1g|-|B\cap F_1g|\bigr|\le |fF_1\triangle F_1|=2|fF_1\setminus F_1|\le2|FF_1\setminus F_1|\le2\eps|F_1|,
$$
and the same for $A$, which yields
\begin{equation}
|B\cap F_1g|-|A\cap F_1g|\ge (\underline D_F(B,A)-4\eps)|F_1|.
\end{equation}
To end the proof, it remains to apply the infimum over all $g\in G$ on the left, and divide both sides by $|F_1|$.
\end{proof}

The first two equalities in the lemma below have been proved in \cite{DHZ}.

\begin{lem}\label{bd}
The values of $\underline D(B)$, $\overline D(B)$ and $\underline D(B,A)$
do not depend on the choice of the F\o lner sequence, the limits superior and inferior in the definition are in fact limits, and moreover
\begin{align*}
\underline D(B) &= \sup_F\ \underline D_F(B),\\
\overline D(B) &= \,\inf_F\ \overline D_F(B),\\
\underline D(B,A) &= \sup_F\ \underline D_F(B,A),
\end{align*}
where $F$ ranges over all finite subsets of $G$.
\end{lem}

\begin{proof}
We will prove the third equation. Then, plugging in $A=\emptyset$ we will get the first equation and passing to the complement $B^c$ we will get the second equation. The inequality
$\limsup_{n\to\infty}\underline D_{F_n}(B,A)\le\sup\{\underline D_F(B,A): F\subset G, F \text{ is finite}\}$
is obvious. It remains to show that
$$
\liminf_{n\to\infty}\underline D_{F_n}(B,A)\ge\sup\{\underline D_F(B,A): F\subset G, F \text{ is finite}\}.
$$
At the same time this will prove the existence of all three limits.

Let $F\subset G$ be a finite set. Given $\eps>0$, for any $n$ large enough $F_n$ is $(F,\eps)$-invariant, hence Lemma \ref{bdc}, implies that $\liminf_{n\to \infty} \underline D_{F_n}(B,A)\ge \underline D_F(B,A)-4\eps$. Since $\eps>0$ is arbitrary, it can be ignored.
\end{proof}

\begin{cor}\label{coro}We have
$$
\underline D(B)-\overline D(A)\le \underline D(B,A).
$$
\end{cor}
\begin{proof} Fix a F\o lner \sq\ $(F_n)$. By the above lemma, we can write
\begin{multline*}
\underline D(B,A)= \lim_{n\to\infty} \ \inf_{g\in G} \frac{|B\cap F_ng|-|A\cap F_ng|}{|F_n|}\ge\\
\lim_{n\to\infty} \ \inf_{g\in G} \frac{|B\cap F_ng|}{|F_n|}- \lim_{n\to\infty} \ \sup_{g\in G}\frac{|A\cap F_ng|}{|F_n|}= \underline D(B)-\overline D(A).
\end{multline*}
\end{proof}

\subsection{Tilings of amenable groups}

In this section we will briefly recall the notion of a quasitiling and tiling of a group and we will quote from \cite{DHZ} the result on the existence of tilings, with arbitrarily good F\o lner properties, of countable amenable groups. These notions and the result will not be used until Section~\ref{cztery},
but due to their generality we put them in the preliminaries. Later, in Section~\ref{piec} we will also introduce the notion of dynamical quasitilings and tilings (and some results on their existence). In contrast to dynamical (quasi)tilings, the (quasi)tilings of this subsection can be regarded as ``static'' or ``algebraic''.

\begin{defn}\label{quasi} A \emph{quasitiling} $\CT$ of a group $G$ is determined by two objects:
\begin{enumerate}
	\item a finite collection $\CS(\CT)$ of finite subsets of $G$ containing the unity $e$,
	called \emph{the shapes};
	\item a finite collection $\CC(\CT) = \{C(S):S\in\CS(\CT)\}$ of disjoint subsets of $G$, called \emph{sets of centers} (for the shapes).
\end{enumerate}
The quasitiling is then the family $\CT=\{(S,c):S\in\CS(\CT),\ c\in C(S)\}$. We require the map $(S,c)\mapsto Sc$ to be injective. Hence, by the \emph{tiles} of \,$\CT$ (denoted by the letter $T$) we will mean either the sets $Sc$ or the pairs $(S,c)$ (i.e., the tiles with defined centers), depending on the context.
\end{defn}
Note that every quasitiling $\CT$ can be represented in a symbolic form, as a point $\CT\in\Delta^G$, with the alphabet $\Delta = \CS(\CT)\cup\{0\}$, as follows: $\CT_g=S \iff g\in C(S)$, $\CT_g=0$ otherwise.

\begin{defn}\label{qt} Let $\eps\in[0,1)$ and $\alpha\in(0,1]$ and let $K\subset G$ be a finite set. A quasitiling $\CT$ is called
\begin{enumerate}
	\item \emph{$(K,\eps)$-invariant} if all shapes of $\CT$ are \emph{$(K,\eps)$-invariant};
	\item \emph{$\eps$-disjoint} if there exists a mapping $T\mapsto T^\circ$ (\,$T\in\CT$) such that
	\begin{itemize}
	\item $T^\circ\subset T$, $\frac{|T^\circ|}{|T|}>1-\eps$ and
	\item $T\neq T'\implies T^\circ\cap {T'}^\circ=\emptyset$;
	\end{itemize}
    \item \emph{disjoint} if the tiles of $\CT$ are pairwise disjoint;
	\item \emph{$\alpha$-covering} if $\underline D(\bigcup\CT)\ge\alpha$;
	\item a \emph{tiling} if it is a partition of $G$.
\end{enumerate}
\end{defn}

One of the most fruitful facts in ergodic theory (as well as \tl\ dynamics) of amenable group actions is the following fact due to Ornstein and Weiss (\cite[Proposition 4]{OW}), which can be reformulated as follows:

\begin{thm}
Let $G$ be a countable amenable group. Then, for any finite set $K$ and any $\eps,\delta,\gamma>0$, there exists a $(K,\eps)$-\inv, $(1-\delta)$-covering, $\gamma$-disjoint quasitiling of $G$.
\end{thm}

The authors of \cite{OW} indicate also that it is possible to create disjoint quasitilings as above.
In \cite{DHZ} we were able to improve the above and replace the quasitilings by tilings
(this seemingly small improvement will become crucial in Section \ref{cztery}).

\begin{thm}{\cite[Theorem 4.3]{DHZ}}\label{ourtilings}
Let $G$ be a countable amenable group. Then, for any finite set $K$ and any $\eps>0$, there exists a $(K,\eps)$-\inv\ tiling of $G$.
\end{thm}

\section{The comparison property}
The key notions of this paper are given below (see also \cite{K}).

\begin{defn}\label{defcom}
Let $G$ be a countable amenable group.
\begin{enumerate}
\item Let $G$ act on a \zd\ compact metric space $X$. For two clopen sets $\mathsf A,\mathsf B\subset X$, we say that $\mathsf A$ is \emph{subequivalent} to $\mathsf B$ (and write $\mathsf A\preccurlyeq \mathsf B$), if there exists a finite partition $\mathsf A=\bigcup_{i=1}^k \mathsf A_i$ of $\mathsf A$ into clopen sets and there are elements $g_1,g_2,\dots,g_k$ of $G$ such that $g_1(\mathsf A_1), g_2(\mathsf A_2),\dots,g_k(\mathsf A_k)$ are disjoint subsets of $\mathsf B$. We say that the action \emph{admits comparison} if for any pair of clopen subsets $\mathsf A,\mathsf B$ of $X$, the condition that for each \im\ $\mu$ on $X$ we have $\mu(\mathsf A)<\mu(\mathsf B)$, implies $\mathsf A\preccurlyeq \mathsf B$.
\item If every action of $G$ on any \zd\ compact metric space admits comparison then we will say that $G$ has the \emph{comparison property}.
\end{enumerate}
\end{defn}
Clearly, $\mathsf A\preccurlyeq \mathsf B$ implies $\mu(\mathsf A)\le\mu(\mathsf B)$ for every invariant measure $\mu$, so comparison is nearly an equivalence between subequivalence and the inequality for all \im s.

\begin{rem}\label{from0}
Let two clopen sets $\mathsf A,\mathsf B$ satisfy $\mu(\mathsf A)<\mu(\mathsf B)$ for every \im\ $\mu$. Because the sets $\mathsf A,\mathsf B$ are clopen, the function $\mu\mapsto\mu(\mathsf B)-\mu(\mathsf A)$ is continuous, and since it is positive on a compact set, it is separated from zero, i.e.,
$$
\inf_{\mu\in\M_G(X)}(\mu(\mathsf B)-\mu(\mathsf A))>0.
$$
\end{rem}

Consider also the following seemingly weaker property:
\begin{defn}\label{defwcom}
The action of a countable amenable group $G$ on a \zd\ compact metric space $X$ admits \emph{weak comparison} if there exists a constant $C\ge 1$ such that for any clopen sets $\mathsf A,\mathsf B\subset X$, the condition $\sup_\mu\mu(\mathsf A)<\frac1C\inf_\mu\mu(\mathsf B)$ (where $\mu$ ranges over all \im s) implies $\mathsf A\preccurlyeq \mathsf B$.
\end{defn}

Clearly, comparison implies weak comparison. We will show that these properties are in fact equivalent.

\begin{lem}\label{compisweakcomp}
Weak comparison implies comparison.
\end{lem}
\begin{proof}
Suppose the action of a countable amenable group $G$ on a \zd\ compact metric space $X$ admits weak comparison with a constant $C$. Let two clopen sets $\mathsf A,\mathsf B$ satisfy $\mu(\mathsf A)<\mu(\mathsf B)$ for every \im\ $\mu$. By
Remark \ref{from0}, $\inf_{\mu\in\M_G(X)}(\mu(\mathsf B)-\mu(\mathsf A))>\eps$ for some positive $\eps$.
We order the group (arbitrarily) by natural numbers, as $G=\{g_1,g_2,\dots\}$ (or $G=\{g_1,\dots,g_n\}$ in case $G$ is finite). We let $\mathsf A_1=\mathsf A\cap g_1^{-1}(\mathsf B)$, and $\mathsf B_1=g_1(\mathsf A_1)$. For each $k>1$ (or $1<k\le n$ in the finite case) we set inductively
$$
\mathsf A_k = \mathsf A\setminus\Bigl(\bigcup_{i=1}^{k-1} \mathsf A_i\Bigr)\cap g_k^{-1}\left(\mathsf B\setminus\Bigl(\bigcup_{i=1}^{k-1} \mathsf B_i\Bigr)\right),
$$
and $\mathsf B_k=g_k(\mathsf A_k)$. It is not hard to see that the sets $\mathsf A_k$ and $\mathsf B_k$ are clopen (some of them possibly empty), disjoint subsets of $\mathsf A$ and $\mathsf B$, respectively and $\mu(\mathsf A_k)=\mu(\mathsf B_k)$ for each $k$ and every \im\ $\mu$. Consider the remainder sets
$$
\mathsf A_0=\mathsf A\setminus\Bigl(\bigcup_{k=1}^{\infty} \mathsf A_k\Bigr)\text{\ \ and \ \ }\mathsf B_0=\mathsf B\setminus\Bigl(\bigcup_{k=1}^{\infty} \mathsf B_k\Bigr),
$$
or in the finite case
$$
\mathsf A_0=\mathsf A\setminus\Bigl(\bigcup_{k=1}^{n} \mathsf A_k\Bigr)\text{\ \ and \ \ }\mathsf B_0=\mathsf B\setminus\Bigl(\bigcup_{k=1}^{n} \mathsf B_k\Bigr).
$$
Clearly, for each \im\ $\mu$ we have $\mu(\mathsf B_0)\ge\eps$. We claim that $\mu(\mathsf A_0)=0$. It suffices to consider an ergodic measure. But if $\mu(\mathsf A_0)$ was positive, then, by ergodicity, there would exist an $x\in \mathsf A_0$ and $g=g_k$ (for some $k$) such that $g_k(x)\in \mathsf B_0$. This is a contradiction, as, by construction, no orbit starting in $\mathsf A_0$ visits the set $\mathsf B_0$. Now, by countable additivity of the measures, we obtain, for each \im\ $\mu$,
$$
\lim_{k\to\infty} \mu\!\left(\mathsf A\setminus\Bigl(\bigcup_{i=1}^k \mathsf A_i\Bigr)\right) =0.
$$
Clearly, the limit is decreasing. Since the measured sets are clopen, the above measure values viewed as functions on the set of \im s are continuous, and thus the convergence is uniform. Let $\delta>0$ be strictly smaller than $\frac\eps C$. Then, for $k$ large enough we have, simultaneously for all \im s $\mu$,
$$
\mu\!\left(\mathsf A\setminus\Bigl(\bigcup_{i=1}^k \mathsf A_i\Bigr)\right)\le\delta<\frac\eps C\le
\frac1C\,\mu\!\left(\mathsf B\setminus\Bigl(\bigcup_{i=1}^k \mathsf B_i\Bigr)\right).
$$
By the weak comparison assumption, we get
$$
\mathsf A\setminus\Bigl(\bigcup_{i=1}^k \mathsf A_i\Bigr)\preccurlyeq \mathsf B\setminus\Bigl(\bigcup_{i=1}^k \mathsf B_i\Bigr),
$$
which, together with the obvious fact that $\bigcup_{i=1}^k \mathsf A_i\preccurlyeq \bigcup_{i=1}^k \mathsf B_i$, completes the proof of $\mathsf A\preccurlyeq \mathsf B$.
\end{proof}

\begin{rem}\label{finitecomp}
The above proof shows also that every finite group $G=\{g_1,g_2,\dots,g_n\}$ has the comparison property. For such a group we have $\mathsf A_0=\mathsf A\setminus(\bigcup_{i=1}^n\mathsf A_i)$. The fact that $\mathsf A_0$ has measure $0$ for all \im s implies that it is empty.
\end{rem}

\begin{rem}\label{disjoint}In the definition of comparison, it suffices to consider only disjoint clopen sets $\mathsf A, \mathsf B$. Indeed, $\{\mathsf A\cap\mathsf B, \mathsf A\setminus\mathsf B\}$ is a clopen partition of $\mathsf A$, and $g_0=e$ sends $\mathsf A\cap\mathsf B$ inside $\mathsf B$, so if $(\mathsf A\setminus\mathsf B)\preccurlyeq(\mathsf B\setminus\mathsf A)$ then also $\mathsf A\preccurlyeq\mathsf B$.
Also note that, for any measure~$\mu$, $\mu(\mathsf A)<\mu(\mathsf B)$ if and only if $\mu(\mathsf A\setminus\mathsf B)<\mu(\mathsf B\setminus\mathsf A)$.
\end{rem}

It is known that many important countable amenable groups, for instance $\z$, $\z^d$,
have the comparison property. However, the following question remains open:

\begin{ques}\label{3.7}
Does every countable amenable group have the comparison property?
\end{ques}

In this paper we will provide a positive answer in a large class of groups.

\section{Banach density interpretation of the comparison property}\label{two}

Now we provide a characterization of the comparison property of a countable amenable group in terms of Banach density advantage for subsets of the group.

\subsection{Passing between clopen subsets of $X$ and subsets of $G$}
This subsection contains fairly standard tools, often exploited in symbolic dynamics. We include them for completeness and as an opportunity to introduce our notation. We continue to assume that $G$ is a countable amenable group.
\subsubsection{}\label{211}
First suppose that $G$ acts on a \zd\ compact metric space in which we have two disjoint clopen sets $\mathsf A$ and $\mathsf B$. Define a map $\pi_{\mathsf A\mathsf B}:X\to\{\mathsf 0,\mathsf 1,\mathsf 2\}^G$ by the formula
$$
(\pi_{\mathsf A\mathsf B}(x))_g=\begin{cases}\mathsf 1\\\mathsf 2\\\mathsf 0\end{cases} \iff g(x)\in \begin{cases} \mathsf A\\\mathsf B\\(\mathsf A\cup \mathsf B)^c,\end{cases}
$$
respectively ($g\in G$). As easily verified, $\pi_{\mathsf A\mathsf B}$ is continuous and intertwines the action on $X$ with the shift action, in other words, it is a \tl\ factor map onto its image $Y_{\mathsf A\mathsf B}=\pi_{\mathsf A\mathsf B}(X)$, which is a subshift,  in which we can distinguish two natural clopen sets, the cylinders $[\mathsf 1]$ and $[\mathsf 2]$. Notice that $\pi_{\mathsf A\mathsf B}^{-1}([\mathsf 1])=\mathsf A$ and $\pi_{\mathsf A\mathsf B}^{-1}([\mathsf 2])=\mathsf B$, hence for every \im\ $\mu$ on $X$ we have $\mu(\mathsf A)=\nu([\mathsf 1])$ and $\mu(\mathsf B)=\nu([\mathsf 2])$, where $\nu=\pi_{\mathsf A\mathsf B}^*(\mu)$ is the ``pushdown'' of $\mu$ onto $Y_{\mathsf A\mathsf B}$ given by $\nu(\cdot) = \mu(\pi_{\mathsf A\mathsf B}^{-1}(\cdot))$. It is well-known that $\pi_{\mathsf A\mathsf B}^*$ is a surjection onto the set $\M_G(Y_{\mathsf A\mathsf B})$ (from now on abbreviated as $\M_{\mathsf A\mathsf B}$) of all shift-\im s on $Y_{\mathsf A\mathsf B}$. For each $x\in X$ we define two subsets of $G$,
\begin{align}
A_x &= \{g:g(x)\in \mathsf A\}=\{g:(\pi_{\mathsf A\mathsf B}(x))_g=\mathsf 1\}=\{g:g(\pi_{\mathsf A\mathsf B}(x))\in[\mathsf 1]\},\label{kiki}\\
B_x &= \{g:g(x)\in \mathsf B\}=\{g:(\pi_{\mathsf A\mathsf B}(x))_g=\mathsf 2\}=\{g:g(\pi_{\mathsf A\mathsf B}(x))\in[\mathsf 2]\}.\label{kiko}
\end{align}

In the above context we can define new notions:
\begin{defn}\label{dbar}
We fix in $G$ a F\o lner \sq\ $(F_n)$. The terms
\begin{align*}
\underline D(\mathsf B)&=\limsup_{n\to\infty}\ \inf_{x\in X}\underline D_{F_n}(B_x),\\
\overline D(\mathsf B)&=\liminf_{n\to\infty}\ \sup_{x\in X}\overline D_{F_n}(B_x),\\
\underline D(\mathsf B,\mathsf A)&=\limsup_{n\to\infty}\ \inf_{x\in X}\underline D_{F_n}(B_x,A_x),
\end{align*}
will be called the \emph{uniform lower Banach density of (visits of the orbits in) $\mathsf B$}, \emph{uniform upper Banach density of $\mathsf B$} and \emph{uniform Banach density advantage of $\mathsf B$ over $\mathsf A$}.
\end{defn}

A statement analogous to Lemma \ref{bd} holds:

\begin{lem}\label{bbb}
The values of $\underline D(\mathsf B)$, $\overline D(\mathsf B)$ and $\underline D(\mathsf B,\mathsf A)$ do not depend on the choice of the F\o lner sequence, the limits superior and inferior in the definition are in fact limits, and moreover
\begin{align*}
\underline D(\mathsf B)&=\sup_F\ \inf_{x\in X}\underline D_F(B_x),\\
\overline D(\mathsf B)&=\inf_F\ \sup_{x\in X}\overline D_F(B_x),\\
\underline D(\mathsf B,\mathsf A)&=\sup_F\ \inf_{x\in X}\underline D_F(B_x,A_x),
\end{align*}
where $F$ ranges over all finite subsets of $G$.
\end{lem}

\begin{proof} The proof is identical to the proof of Lemma \ref{bd}, with the only difference that Lemma \ref{bdc} applies to the sets $A_x,B_x$ whenever $F_n$ is $(F,\eps)$-\inv, simultaneously for all $x\in X$.
\end{proof}

In a moment we will connect the above notions with the values assumed by the invariant measures on $X$ on the sets $\mathsf A$ and $\mathsf B$.

\subsubsection{}\label{212}
We will now describe the opposite passage: from subsets of $G$ to clopen subsets of some \zd\ compact metric space on which we have a $G$-action.  Suppose we have two disjoint subsets $A$ and $B$ of $G$. Then they determine an element $y^{AB}$ of the symbolic space $\{\mathsf 0,\mathsf 1,\mathsf 2\}^G$, given by the rule
$$
y^{AB}_g=\begin{cases}\mathsf 1\\\mathsf 2\\\mathsf 0\end{cases} \iff g\in \begin{cases} A\\B\\(A\cup B)^c,\end{cases}
$$
respectively ($g\in G$).
The shift-orbit closure of $y^{AB}$, i.e., the set
$$
Y^{AB}=\overline{\{g(y^{AB}):g\in G\}}
$$
is a subshift, which we will call the \emph{subshift associated with the sets $A, B$}. The set of its \im s, $\M_G(Y^{AB})$, will be abbreviated as $\M^{AB}$. In this subshift we will distinguish two clopen sets, $\mathsf A=[\mathsf 1]$ and $\mathsf B=[\mathsf 2]$. It is almost immediate to see that if we apply the definitions of the preceding paragraph to the shift action on $Y^{AB}$ and the above sets $\mathsf A,\mathsf B$ then the factor map $\pi_{\mathsf A\mathsf B}$ is the identity, and $A_{y^{AB}}=\{g:y^{AB}_g=\mathsf 1\} = A$ and $B_{y^{AB}}=\{g:y^{AB}_g=\mathsf 2\}=B$.

\begin{prop}\label{prop}
\begin{enumerate}
\item Suppose $G$ acts on a \zd\ compact metric space $X$ in which we are given two disjoint clopen sets, $\mathsf A,\mathsf B$. Then
\begin{align*}
\inf_{\mu\in\M_G(X)}\mu(\mathsf B)&=\underline D(\mathsf B)=\inf_{x\in X}\underline D(B_x), \\
\sup_{\mu\in\M_G(X)}\mu(\mathsf B)&=\overline D(\mathsf B)=\sup_{x\in X}\overline D(B_x),\\
\inf_{\mu\in\M_G(X)}(\mu(\mathsf B)-\mu(\mathsf A))&=\underline D(\mathsf B,\mathsf A)=\inf_{x\in X}\underline D(B_x,A_x).
\end{align*}
\item
Next suppose that $A$ and $B$ are disjoint subsets of $G$. Consider the cylinders $[\mathsf 1]$ and $[\mathsf 2]$ in the subshift $Y^{AB}$ associated with these sets. Then
\begin{gather*}
\inf_{\mu\in\M^{AB}}\mu([2])=\underline D(B),\\
\sup_{\mu\in\M^{AB}}\mu([2])=\overline D(B),\\
\inf_{\mu\in\M^{AB}}(\mu([2])-\mu([1]))=\underline D(B,A).
\end{gather*}
\end{enumerate}
\end{prop}

\begin{proof}
In (1) we will only show the last line of equalities. The first line will then follow by plugging in $\mathsf A=\emptyset$ and the the second one by considering the complement of $\mathsf B$. First suppose that we have sharp inequality $\inf_{\mu\in\M_G(X)}(\mu(\mathsf B)-\mu(\mathsf A))>\underline D(\mathsf B,\mathsf A)$. By Lemma \ref{bbb}, there exists an $\eps>0$ such that for every finite set $F$, $\inf_{\mu\in\M_G(X)} (\mu(\mathsf B)-\mu(\mathsf A))-\eps > \inf_{x\in X} \underline D_F(B_x,A_x)$.
In particular for every set $F_n$ in an \emph{a priori} selected F\o lner \sq, there exists some $x_n\in X$ and $g_n\in G$ with
$$
\inf_{\mu\in\M_G(X)}(\mu(\mathsf B)-\mu(\mathsf A))-\eps > \frac1{|F_n|}(|B_{x_n}\cap F_ng_n|-|A_{x_n}\cap F_ng_n|).
$$
Note that $|B_{x_n}\cap F_ng_n|=|\{f\in F_n:fg_n(x_n)\in \mathsf B\}|$ (and analogously for $\mathsf A$),
thus the right hand side takes on the form
$$
\frac1{|F_n|}(|\{f\in F_n:fg_n(x_n)\in \mathsf B\}|-|\{f\in F_n:fg_n(x_n)\in \mathsf A\}|).
$$
The function $\mathsf W\mapsto \frac1{|F_n|}|\{f\in F_n:fg_n(x_n)\in \mathsf W\}|$ defined on Borel subsets of $X$ is equal to the probability measure $\frac1{|F_n|}\sum_{f\in F_n}\delta_{fg_n(x_n)}$. This \sq\ of measures has a sub\sq\ convergent in the weak-star topology to some $\mu_0\in\M_G(X)$. Since the characteristic functions of the clopen sets $\mathsf A, \mathsf B$ are continuous, we have
$$
\inf_{\mu\in\M_G(X)}(\mu(\mathsf B)-\mu(\mathsf A))-\eps\ge \mu_0(\mathsf B)-\mu_0(\mathsf A),
$$
which is a contradiction. We have proved that $\inf_{\mu\in\M_G(X)}(\mu(\mathsf B)-\mu(\mathsf A))\le\underline D(\mathsf B,\mathsf A)$. The inequality $\underline D(\mathsf B,\mathsf A)\le\inf_{x\in X}\underline D(B_x,A_x)$ is trivial; both sides differ by changing the order of $\limsup_n$ and $\inf_x$ and on the left the infimum is applied earlier.

For the last missing inequality, notice that, given $\eps>0$, there exists an ergodic measure $\mu_0\in\M_G(X)$ with $\mu_0(\mathsf B)-\mu_0(\mathsf A) < \inf_{\mu\in\M_G(X)}(\mu(\mathsf B)-\mu(\mathsf A))+\eps$. There exists a point $x\in X$ generic for $\mu_0$. Then $y=\pi_{\mathsf A\mathsf B}(x)$ is generic for $\nu_0=\pi^*_{\mathsf A\mathsf B}(\mu_0)$. This implies that
$\frac1{|F_n|}|\{f\in F_n:f(y)\in[\mathsf 1]\}|\to \nu_0([\mathsf 1])=\mu_0(\mathsf A)$ (and analogously for $[\mathsf 2]$ and $\mathsf B$). Thus, for each sufficiently large $n$ we have
$$
\frac1{|F_n|}(|\{f\in F_n:f(y)\in[\mathsf 2]\}|-|\{f\in F_n:f(y)\in[\mathsf 1]\}|)< \mu_0(\mathsf B)-\mu_0(\mathsf A)+\eps.
$$
But $f(y)\in[\mathsf 1]\iff(\pi_{\mathsf A\mathsf B}(x))_f=1\iff f(x)\in \mathsf A\iff f\in A_x$ (and analogously for $\mathsf B$), so we have shown that
$$
\frac1{|F_n|}(|B_x\cap F_n|-|A_x\cap F_n|)< \inf_{\mu\in\M_G(X)}(\mu(\mathsf B)-\mu(\mathsf A))+2\eps.
$$
Clearly, the left hand side is not smaller than
$$
\inf_{g\in G}\frac1{|F_n|}(|B_x\cap F_ng|-|A_x\cap F_ng|)=\underline D_{F_n}(B_x,A_x).
$$
Passing to the limit over $n$ and then applying infimum over all $x\in X$ we obtain
$\inf_{x\in X}\underline D(B_x,A_x)\le \inf_{\mu\in\M_G(X)}(\mu(\mathsf B)-\mu(\mathsf A))+2\eps$.
Since this is true for every $\eps>0$, (1) is proved.

\smallskip
We pass to proving (2). As before, the last equality suffices. From (1) applied to the cylinders $\mathsf A=[\mathsf 1]$ and $\mathsf B=[\mathsf 2]$ we get
$$
\inf_{\mu\in\M^{AB}}(\mu([2])-\mu([1]))=\underline D([\mathsf 2],[\mathsf 1])
=\lim_{n\to\infty}\ \,\inf_{y\in Y^{AB}}\ \inf_{g\in G} \frac1{|F_n|} (|B_y\cap F_ng|-|A_y\cap F_ng|).
$$
The above difference $|B_y\cap F_ng|-|A_y\cap F_ng|$ depends on the block $y|_{F_ng}$.
Notice that we are considering a transitive subshift with the transitive point $y^{AB}$ (i.e., whose orbit is dense in the subshift), so every block $y|_{F_ng}$ (for any $y\in Y^{AB}$ and any $g\in G$) occurrs also in $y^{AB}$ as a block $y^{AB}|_{F_ng'}$ for some $g'$ (the converse need not be true, unless $y$ is another transitive point).
Thus, for any $n$, the infimum over $y\in Y^{AB}$ on the right hand side of the formula displayed above is the smallest for $y=y^{AB}$. Recall that $A_{y^{AB}}=A$ and $B_{y^{AB}}=B$. We have proved that
$$
\inf_{\mu\in\M^{AB}}(\mu([2])-\mu([1])) = \lim_{n\to\infty}\ \inf_{g\in G}\frac1{|F_n|} (|B\cap F_ng|-|A\cap F_ng|).
$$
The right hand side is precisely $\underline D(B,A)$.
\end{proof}

The following notions are standard in symbolic dynamics. We assume that $G$ is a countable group (in the remainder of this subsection amenability is inessential).

\begin{defn} Let $\Lambda$ and $\Delta$ be some finite sets (alphabets). By a \emph{block code} we will mean any function $\Xi:\Lambda^F\to\Delta$, where $F$ is a finite subset of $G$
(called the \emph{coding horizon} of \,$\Xi$).
\end{defn}

The Curtis--Hedlund--Lyndon Theorem \cite{H} (which holds for actions of any countable group) states:
\begin{thm}\label{CHL}
Let $X\subset\Lambda^G$ be a subshift (over some finite alphabet $\Lambda$). Let
$\Delta$ be a finite set. Then $\xi:X\to\Delta^G$ is a topological factor map (i.e., a continuous and shift-equivariant map, the image is then a subshift over $\Delta$) if and only if there exists a finite set $F\subset G$ and a block code $\Xi:\Lambda^F\to\Delta$,
such that, for all $x\in X$ and $g\in G$ we have the equality
$$
(\xi(x))_g = \Xi(g(x)|_F).
$$
\end{thm}
The term ``block code'' refers to both $\Xi$ and $\xi$, depending on the context, and $F$ is called a coding horizon of $\xi$ (and of $\Xi$). Clearly, if $F$ is a coding horizon of $\xi$ (and of $\Xi$), so is any finite set containing $F$.
\smallskip

\begin{defn}\label{local rule}
Let $X\subset\Lambda^G$ be a subshift. For each $x\in X$ let $A_x\subset G$ and let $\tilde\varphi_x: A_x\to G$ be some function. For $X'\subset X$, we will say that the family $\{\tilde\varphi_x\}_{x\in X'}$ is \emph{determined by a block code} if there exists a block code $\Xi:\Lambda^F\to E$, where $E$ is a finite subset of $G$ (and so is $F$), such that if we denote
$$
\varphi_x(g) = \Xi(g(x)|_F),
$$
($x\in X, g\in G$),
then, for each $x\in X'$, the mapping from $A_x$ to $G$, defined by
$$
a\mapsto \varphi_x(a)a,
$$
($a\in A_x$), coincides with $\tilde\varphi_x$. The elements $\varphi_x(a)$ (belonging to $E$) will be called the \emph{multipliers} of $\tilde\varphi_x$.
\end{defn}

A simple way of checking, that a family $\{\tilde\varphi_x\}_{x\in X'}$ is determined by a block code, is finding a finite set $F$ such that,
for any $x_1,x_2\in X'$ and $a_1\in A_{x_1}, a_2\in A_{x_2}$,
\begin{equation}\label{clr}
a_1(x_1)|_F = a_2(x_2)|_F \ \implies \
\tilde\varphi_{x_1}(a_1)a_1^{-1} = \tilde\varphi_{x_2}(a_2){a_2}^{-1}.
\end{equation}

The following theorem connects the above definition with the relation of subequivalence.

\begin{thm}\label{tutka}
\begin{enumerate}
\item Let $X\subset\Lambda^G$ be a subshift. Consider the pair of disjoint clopen subsets $\mathsf A, \mathsf B\subset X$. Then $\mathsf A\preccurlyeq\mathsf B$ if and only if there exists a family of functions $\tilde\varphi_x:G\to G$ determined by a block code, such that for all $x\in X$, $\tilde\varphi_x$ restricted to $A_x=\{g:g(x)\in\mathsf A\}$ is an injection to $B_x=\{g:g(x)\in\mathsf B\}$.
\item If, moreover, $X$ is transitive with a transitive point $x^*$, then the above condition $\mathsf A\preccurlyeq\mathsf B$ is equivalent to the existence of just one function $\tilde\varphi_{x^*}$ determined by a block code, whose restriction to $A_{x^*}$ is an injection to $B_{x^*}$.
\end{enumerate}
\end{thm}

\begin{proof}
(1) Firstly suppose that $\mathsf A\preccurlyeq\mathsf B$. Let $\{\mathsf A_1,\mathsf A_2,\dots,\mathsf A_k\}$ be the clopen partition of $\mathsf A$ and let $g_1,g_2,\dots, g_k$ be the elements of $G$ such that the sets $\mathsf B_i = g_i(\mathsf A_i)$ are disjoint subsets of $\mathsf B$. Let $E = \{g_1,g_2,\dots,g_k\}$. Consider the mapping $\xi:X\to E^G$ given by the following rule
$$
(\xi(x))_g = \begin{cases}
g_i &\text{ if \ }g(x)\in\mathsf A_i, \ i=1,2,\dots,k,\\
g_1 &\text{ otherwise},
\end{cases}
$$
($g\in G$).
Since the sets $\mathsf A_i$ and $X\setminus \mathsf A$ are clopen in $X$, the above map is continuous and, as easily verified, it is shift-equivariant. Thus, it is a \tl\ factor map from $X$ into $E^G$. By Theorem \ref{CHL}, there exists a block code $\Xi:\Lambda^F\to E$ (with some finite coding horizon $F$) satisfying, for all $x\in X$ and $g\in G$, the equality
$$
(\xi(x))_g = \Xi(g(x)|_F).
$$
For each $x\in X$ we define $\varphi_x:G\to E$ by $\varphi_x(g)= (\xi(x))_g$ and $\tilde\varphi_x:G\to G$ by $\tilde\varphi_x(g)=\varphi_x(g)g$, i.e., the family of maps $\{\tilde\varphi_x\}_x$ is determined by the block code $\Xi$. We need to show that, for every $x\in X$, $\tilde\varphi_x$ restricted to $A_x$ is an injection to $B_x$.

Throughout this paragraph we fix some $x\in X$ and skip the subscript $x$ in the writing of $A_x, B_x$, $\varphi_x$ and $\tilde\varphi_x$. For $i=1,2,\dots,k$ let $A_i = A\cap\varphi^{-1}(g_i)$. Clearly, $\{A_1,A_2,\dots,A_k\}$ is a partition of $A$ and for every $a\in A$ we have:
$$
a\in A_i \iff \varphi(a)=g_i \iff (\xi(x))_a = g_i \iff  a(x) \in \mathsf A_i, \ \  (i=1,2,\dots,k).
$$
Further, $a(x)\in\mathsf A_i$ yields $g_ia(x)\in\mathsf B_i\subset \mathsf B$, which implies that $g_ia\in B$. Since $g_ia = \varphi(a)a = \tilde\varphi(a)$, we have shown that $\tilde\varphi$ sends $A$ into $B$. For injectivity of the restriction $\tilde\varphi|_A$, observe that if $a_1\neq a_2$ and both elements belong to the same set $A_i$ then their images by $\tilde\varphi$, equal to $g_ia_1$ and $g_ia_2$, respectively, are different by cancellativity. If $a_1\in A_i$ and $a_2\in A_j$ with $i\neq j$, then
$\tilde\varphi(a_1)(x)=g_ia_1(x)\in\mathsf B_i$ and $\tilde\varphi(a_2)(x)=g_ja_2(x)\in\mathsf B_j$. Since $\mathsf B_i$ and $\mathsf B_j$ are disjoint, the elements $\tilde\varphi(a_1)$ and $\tilde\varphi(a_2)$ must be different.

\smallskip
Now suppose that there exist injections $\tilde\varphi_x:A_x\to B_x$ (for all $x\in X$) determined by a block code $\Xi:\Lambda^F\to E=\{g_1,g_2,\dots,g_k\}\subset G$, where the elements $g_i$ are written without repetitions, i.e., are different for different indices $i=1,2,\dots,k$. That is, denoting, for each $g\in G$,
$$
\varphi_x(g) = \Xi(g(x)|_F),
$$
we obtain maps $\varphi_x$ such that $g\mapsto\varphi_x(g)g$ restricted to $A_x$ coincides with $\tilde\varphi_x$. Now, for each $i=1,2,\dots,k$ we define
$$
\mathsf A_i=\mathsf A\cap [\Xi^{-1}(g_i)]=\{x\in\mathsf A:\Xi(x|_F)=g_i\}.
$$
Clearly, $\{\mathsf A_1,\mathsf A_2,\dots,\mathsf A_k\}$ is a clopen partition of $\mathsf A$. Let $x\in \mathsf A_i$ (for some $i=1,2,\dots, k$). Then $e\in A_x$ and thus $\tilde\varphi_x(e)\in B_x$, i.e., $\tilde\varphi_x(e)(x)\in\mathsf B$. But $\tilde\varphi_x(e)=\varphi_x(e)=\Xi(x|_F)=g_i$. We have shown that $g_i(\mathsf A_i)\subset\mathsf B$.

It remains to show that the sets $g_i(\mathsf A_i)$ are disjoint. Suppose that for some $i\neq j$ there exists $x\in X$ belonging to both $g_i(\mathsf A_i)$ and $g_j(\mathsf A_j)$.
This implies that $g_i^{-1}$ and $g_j^{-1}$ both belong to $A_x$, and $\varphi_x(g_i^{-1}) = g_i$, $\varphi_x(g_j^{-1}) = g_j$. But then
$$
\tilde\varphi_x(g_i^{-1})= \varphi_x(g_i^{-1}) g^{-1}_i= g_ig^{-1}_i=e\text{ \ and \ }\tilde\varphi_x(g_j^{-1})=\varphi_x(g_j^{-1}) g^{-1}_j= g_jg^{-1}_j=e,
$$
which contradicts the injectivity of $\tilde\varphi_x$ on $A_x$.

\smallskip
(2) In view of (1), it suffices to show that if a block code $\Xi:\Lambda^F\to E$ determines an injection $\tilde\varphi_{x^*}:A_{x^*}\to B_{x^*}$ then it also determines (as usual, by the formulas $\varphi_x(g)=\Xi(g(x)|_F)$ and $\tilde\varphi_x(a)=\varphi_x(a)a$\,)
injections $\tilde\varphi_x:A_x\to B_x$ for all $x\in X$.
Fix some $x\in X$ and let $a_1\neq a_2$ belong to $A_x$, i.e., $a_1(x), a_2(x)\in\mathsf A$. Since $x^*$ is a transitive point,
a point $g(x^*)$ (for some $g\in G$) is so close to $x$ that:
\begin{enumerate}
	\item[(a)] $a_1g(x^*),\ a_2g(x^*)\in\mathsf A$,
	\item[(b)] the blocks $g(x^*)|_{Fa_1\cup Fa_2}$ and $x|_{Fa_1\cup Fa_2}$ are equal,
	\item[(c)] $(\forall f\in Ea_1\cup Ea_2) \ \ fg(x^*)\in\mathsf B \iff f(x)\in\mathsf
	B$.
\end{enumerate}
By (a), both $a_1g$ and $a_2g$ belong to $A_{x^*}$. Thus $\tilde\varphi_{x^*}(a_1g)$
and $\tilde\varphi_{x^*}(a_2g)$ are \emph{different} elements of $B_{x^*}$. But
$$
\tilde\varphi_{x^*}(a_1g)=\varphi_{x^*}(a_1g)a_1g\text{ \ \ and \ \ }\tilde\varphi_{x^*}(a_2g)=\varphi_{x^*}(a_2g)a_2g,
$$
which, after canceling $g$, yields
$$
\varphi_{x^*}(a_1g)a_1\neq\varphi_{x^*}(a_2g)a_2.
$$
On the other hand, by (b), $x|_{Fa_1} = g(x^*)|_{Fa_1}$, whence $a_1(x)|_F = a_1g(x^*)|_F$, and
$$
\varphi_{x}(a_1)=\Xi(a_1(x)|_F) = \Xi(a_1g(x^*)|_F) = \varphi_{x^*}(a_1g),
$$
which means that $\tilde\varphi_{x}(a_1)=\varphi_{x}(a_1)a_1=\varphi_{x^*}(a_1g)a_1$.
Analogously, $\tilde\varphi_{x}(a_2)=\varphi_{x^*}(a_2g)a_2$. We have shown that
$\tilde\varphi_x(a_1)\neq \tilde\varphi_x(a_2)$, i.e., $\tilde\varphi_x$ restricted to $A_x$ is injective.

Further, the fact that $\tilde\varphi_{x^*}(a_1g)\in B_{x^*}$ yields
$$
\mathsf B\ni\tilde\varphi_{x^*}(a_1g)(x^*)=\varphi_{x^*}(a_1g)a_1g(x^*)=
\varphi_{x}(a_1)a_1g(x^*).
$$
Since $\varphi_x(a_1)a_1\in Ea_1$, by (c) we get
$$
\mathsf B\ni\varphi_{x}(a_1)a_1(x)=\tilde\varphi_x(a_1)(x),
$$
and hence $\tilde\varphi_x(a_1)\in B_x$. We have shown that $\tilde\varphi_x$ sends $A_x$ injectively to $B_x$.
\end{proof}

\subsection{Banach density comparison property of a group}
\begin{defn}
We say that $G$ has the \emph{Banach density comparison property} if whenever
$A\subset G$ and $B\subset G$ are disjoint and satisfy $\underline D(B,A)>0$ then,
in the subshift \,$Y^{AB}$ there exists an injection $\tilde\varphi:A\to B$ determined by a block code (recall that $y^{AB}$ is a transitive point in $Y^{AB}$ and $A=A_{y^{AB}},\ B=B_{y^{AB}}$, so the above condition is the same as that in Theorem \ref{tutka} (2)).
\end{defn}

\begin{rem}
It is immediate to see that any finite group has the Banach density comparison property.
\end{rem}


We can now completely characterize the comparison property of a countable amenable group in terms of the Banach density comparison property.

\begin{thm}\label{ujowe}
A countable amenable group $G$ has the comparison property if and only if it has the Banach density comparison property.
\end{thm}

\begin{proof}
The theorem holds trivially for finite groups, so we can restrict to infinite groups $G$. Assume that $G$ has the comparison property and let $A,B\subset G$ be disjoint and satisfy $\underline D(B,A)>0$. Then, by Proposition \ref{prop} (2), taking in the subshift $Y^{AB}$ the clopen sets: $\mathsf A=[\mathsf 1]$ and $\mathsf B=[\mathsf 2]$, we have $\inf_{\mu\in\M^{AB}}(\mu(\mathsf B)-\mu(\mathsf A))>0$. By the assumption, $\mathsf A\preccurlyeq\mathsf B$. Now, a direct application of Theorem \ref{tutka} (2) completes the proof of the Banach density comparison property.

Let us pass to the proof of the opposite implication.
Suppose that a countable amenable group $G$ having the Banach density comparison property acts on a \zd\ compact metric space $X$, in which we have selected two clopen sets $\mathsf A$ and $\mathsf B$ satisfying, for each \im\ $\mu$ on $X$, the inequality $\mu(\mathsf A)<\mu(\mathsf B)$.
By Remark \ref{disjoint}, we can assume that $\mathsf A$ and $\mathsf B$ are disjoint; and
by Remark \ref{from0}, we have $\inf_{\mu\in\M_G(X)}(\mu(\mathsf B)-\mu(\mathsf A))>0$. This translates to $\inf_{\nu\in\M_{\mathsf A\mathsf B}}(\nu([\mathsf 2])-\nu([\mathsf 1]))>0$ in the factor subshift $Y_{\mathsf A\mathsf B}$. By Proposition \ref{prop} (1) applied to this subshift, we get $\underline D([\mathsf 2],[\mathsf 1])>0$.

Since we intend to use the Banach density comparison property and Theorem~\ref{tutka}~(2), we need to embed $Y_{\mathsf A\mathsf B}$ in a transitive subshift $Y$ (over the alphabet $\{\mathsf0,\mathsf1,\mathsf2\}$). We also desire a transitive point $y^*$ which satisfies
$\underline D(B_{y^*},A_{y^*})>0$. Below we present the construction of such a transitive subshift.

Choose some positive $\gamma<\underline D([\mathsf 2],[\mathsf 1])$. Fix an increasing (w.r.t. set inclusion) F\o lner \sq\ $(F_n)$ such that $\bigcup_{n=1}^\infty F_n=G$. By choosing a sub\sq\ we can assume that $\sum_{i=1}^{n-1}|F_i|<\frac{1-\gamma}2|F_n|$ for every $n$ (in this place we use the assumption that $G$ is infinite). Next, we need to find a \sq\ of blocks $\mathbf B_n\in\{\mathsf0,\mathsf1,\mathsf2\}^{F_n}$ each appearing as $y_n|_{F_n}$ in some $y_n\in Y_{\mathsf A\mathsf B}$, such that every $y\in Y_{\mathsf A\mathsf B}$ is a coordinatewise limit of a subsequence $\mathbf B_{n_k}$ of the selected blocks. Finally, we need to find a \sq\ $g_n$ of elements of $G$ such that the sets $F_nF_n^{-1}F_ng_n$ are disjoint. All the above steps are possible and easy. Once they are completed, $y^*$ is defined by the rule: for each $n$ and $f\in F_n$ we put $y^*_{fg_n}=\mathbf B_n (f)$, and for all $g$ outside the union $\bigcup_{n=1}^\infty F_ng_n$, we put $y^*_g = \mathsf 2$. We let $Y$ be the closure of the orbit of $y^*$.

The following properties hold:
\begin{itemize}
	\item $Y\supset Y_{\mathsf A\mathsf B}$,
	\item $\underline D(B_{y^*},A_{y^*})\ge \gamma>0$.
\end{itemize}
The first property is obvious by construction: each $y\in Y_{\mathsf A\mathsf B}$ is the limit of a sequence of blocks $\mathbf B_{n_k}$, hence it is also the limit of the \sq\ of elements $g_{n_k}(y^*)$, and thus it belongs to $Y$.

We need to prove the latter property. By the definition of $\underline D([\mathsf 2],[\mathsf 1])$ in the subshift $Y_{\mathsf A\mathsf B}$, there exist arbitrarily large indices $n_k$ such that
\begin{equation}\label{noco}
|\{f\in F_{n_k}:y_{fg}=\mathsf 2\}|-|\{f\in F_{n_k}:y_{fg}=\mathsf 1\}|\ge\gamma|F_{n_k}|,
\end{equation}
for all $y\in Y_{\mathsf A\mathsf B}$ and $g\in G$. It suffices to show an analogous property for $y^*$.

Fix some $g\in G$ and observe the block $y^*|_{F_{n_k}g}$. The set $F_{n_k}g$ either does not intersect any of the sets $F_mg_m$ with $m\ge n_k$ or intersects one of them (say $F_{m_0}g_{m_0}$ with $m_0\ge n_k$).

In the first case, the block $y^*|_{F_{n_k}g}$ consists mostly of symbols $\mathsf 2$; as all symbols different from $\mathsf 2$ appear in $y^*$ only over the intersection of $F_{n_k}g$ with the union of the sets $F_ig_i$ with $i<n_k$, the percentage of such symbols in $y^*|_{F_{n_k}g}$ is at most
$$
\frac1{|F_{n_k}g|}\sum_{i=1}^{n_k-1}|F_ig_i|=\frac1{|F_{n_k}|}\sum_{i=1}^{n_k-1}|F_i|<\frac{1-\gamma}2.
$$
Thus, in this case we have
\begin{equation}\label{noco1}
|\{f\in F_{n_k}:y^*_{fg}=\mathsf 2\}|-|\{f\in F_{n_k}:y^*_{fg}=\mathsf 1\}|\ge\gamma|F_{n_k}|.
\end{equation}
In the latter case, we have $g\in F_{n_k}^{-1}F_{m_0}g_{m_0}$, hence $F_{n_k}g\subset F_{n_k}F_{n_k}^{-1}F_{m_0}g_{m_0}\subset F_{m_0}F_{m_0}^{-1}F_{m_0}g_{m_0}$. By disjointness of the sets $F_nF_n^{-1}F_ng_n$, $F_{n_k}g$ does not intersect any set $F_nF_n^{-1}F_ng_n$ (and hence also $F_ng_n$) with $n\neq m_0$. We will compare the block $y^*|_{F_{n_k}g}$ with the block $y_{m_0}|_{F_{n_k}gg_{m_0}^{-1}}$. We can write
$$
F_{n_k}g = (F_{n_k}g\cap F_{m_0}g_{m_0}) \cup (F_{n_k}g\setminus F_{m_0}g_{m_0}),
$$
and likewise
$$
F_{n_k}gg_{m_0}^{-1} = (F_{n_k}gg_{m_0}^{-1}\cap F_{m_0})\cup (F_{n_k}gg_{m_0}^{-1}\setminus F_{m_0}).
$$
By the definition of $y^*$, the block $y^*|_{F_{n_k}g\cap F_{m_0}g_{m_0}}$ is identical to $y_{m_0}|_{F_{n_k}gg_{m_0}^{-1}\cap F_{m_0}}$, while $y^*|_{F_{n_k}g\setminus F_{m_0}g_{m_0}}$ contains just the symbols $\mathsf 2$. Thus
the difference
$$
|\{f\in F_{n_k}:y^*_{fg}=\mathsf 2\}|-|\{f\in F_{n_k}:y^*_{fg}=\mathsf 1\}|
$$
is not smaller than
$$
|\{f\in F_{n_k}: (y_{m_0})_{fgg_{m_0}^{-1}}=\mathsf 2\}|-|\{f\in F_{n_k}:(y_{m_0})_{fgg_{m_0}^{-1}}=\mathsf 1\}|.
$$
Since $y_{m_0}\in Y_{\mathsf A\mathsf B}$, \eqref{noco} implies that the latter expression is at least $\gamma|F_{n_k}|$. We have proved \eqref{noco1} also in this case.

We have proved that $\underline D(B_{y^*},A_{y^*})\ge\gamma>0$. Now, the Banach density comparison property of $G$ implies that there exists an injection $\tilde\varphi$ from $A_{y^*}$ to $B_{y^*}$ determined by a block code. Thus, by Theorem~\ref{tutka} (2), we get $[\mathsf 1]\preccurlyeq[\mathsf 2]$ in the transitive subshift $Y$, and by restriction to a closed \inv\ set the same holds in $Y_{\mathsf A\mathsf B}$, which, by an application of $\pi_{\mathsf A\mathsf B}^{-1}$, translates to $\mathsf A\preccurlyeq\mathsf B$ in~$X$.
\end{proof}

\subsection{Comparison property via finitely generated subgroups}

\begin{lem}\label{supinf}
Let $G$ act on a \zd\ compact metric space $X$. Let $\mathsf A,\mathsf B\subset X$ be two disjoint clopen sets. Then
$$
\sup_H \inf_{\mu\in\M_H(X)}(\mu(\mathsf B)-\mu(\mathsf A))=\sup_{H'} \inf_{\mu\in\M_{H'}(X)}(\mu(\mathsf B)-\mu(\mathsf A))=\inf_{\mu\in\M_G(X)}(\mu(\mathsf B)-\mu(\mathsf A))
$$
where $H$ ranges over all finitely generated subgroups of $G$ and $H'$ ranges over all subgroups of $G$.
\end{lem}

\begin{proof}
The inequality $\le$ on the left hand side is trivial, while the second inequality $\le$ follows easily from the fact that every measure invariant under the action of $G$ is invariant under the action of $H'$ for any subgroup $H'$ of $G$.

We need to prove the last missing inequality.
By Proposition \ref{prop} (1), we have $\inf_{\mu\in\M_G(X)}(\mu(\mathsf B)-\mu(\mathsf A))=\underline D(\mathsf B,\mathsf A)$. Then, for any positive $\delta$, there exists a finite set $F$ such that
$$
\frac1{|F|}(|B_x\cap Fg|-|A_x\cap Fg|)>\underline D(\mathsf B,\mathsf A)-\delta
$$
for every $x\in X$ and all $g\in G$, in particular for all $g\in H$, where $H$ is the subgroup generated by $F$. Thus, for every $x\in X$, we have
$$
\inf_{g\in H}\frac1{|F|}(|B_x\cap Fg|-|A_x\cap Fg|)\ge\underline D(\mathsf B,\mathsf A)-\delta.
$$
Since $F\subset H$ and $g\in H$, we have $A_x\cap Fg = (A_x\cap H)\cap Fg$. Note that $A_x\cap H$ equals the set $A_x$ defined for the induced action of $H$ on $X$ (and analogously for $B_x$). Thus, the expression on the left hand side above equals $\underline D_F(B_x,A_x)$ evaluated for the action of $H$ on $X$. Now, Lemma \ref{bd} implies $\underline D(B_x,A_x)\ge\underline D(\mathsf B,\mathsf A)-\delta$ for every $x\in X$ (where $\underline D(B_x,A_x)$ is evaluated for the action of $H$ on $X$, and $\underline D(\mathsf B,\mathsf A)$ is evaluated for the action of $G$ on $X$), and Proposition \ref{prop} (1) yields
$$
\inf_{\mu\in\M_H(X)}(\mu(\mathsf B)-\mu(\mathsf A))\ge\underline D(\mathsf B,\mathsf A)-\delta=\inf_{\mu\in\M_G(X)}(\mu(\mathsf B)-\mu(\mathsf A))-\delta.
$$
After applying the supremum over $H$ on the left we can ignore $\delta$ on the right.
\end{proof}

\begin{prop}\label{44}
A countable amenable group $G$ has the comparison property if every finitely generated subgroup $H$ of $G$ has it.
\end{prop}

\begin{proof}
Let $G$ act on a \zd\ compact metric space $X$ and let $\mathsf A,\mathsf B\subset X$ be two disjoint clopen sets satisfying $\underline D(\mathsf B,\mathsf A)>0$. By the preceding lemma (and by Proposition \ref{prop} (1) used twice), there exists a finitely generated subgroup $H$ of $G$ such that the inequality $\underline D(\mathsf B,\mathsf A)>0$ holds also if $\underline D$ is evaluated for the action of $H$. By the comparison property of $H$, we get that $\mathsf A\preccurlyeq\mathsf B$ in this latter action. But this clearly implies the same subequivalence in the action by $G$.
\end{proof}

\begin{rem}
By the proof of Lemma \ref{supinf}, if $(H_n)$ is an increasing \sq\ of subgroups of $G$ such that  $G=\bigcup_{n=1}^\infty H_n$ then
$$
\inf_{\mu\in\M_G(X)}(\mu(\mathsf B)-\mu(\mathsf A))=\lim_{n\to\infty}\inf_{\mu\in\M_{H_n}(X)}(\mu(\mathsf B)-\mu(\mathsf A)).
$$
Thus, in Proposition \ref{44}, the assumption can be weakened to the existence of
an increasing \sq\ $(H_n)$ of subgroups of $G$ such that $G=\bigcup_{n=1}^\infty H_n$, and every $H_n$ has the comparison property.
\end{rem}

\begin{rem}
The converse implication in Proposition \ref{44} is a bit mysterious. On the one hand, since there are no examples of countable amenable groups without the comparison property, clearly, there is no counterexample for the implication in question. On the other hand, we failed to deduce the comparison property of a subgroup of $G$ from the comparison property of the group $G$.
\end{rem}

\section{Comparison property of subexponential groups}\label{cztery}

This section contains our main result: every subexponential group has the comparison property. The theorem is preceded by a few key definitions and lemmas.

\subsection{Correction chains}
We now introduce the key tool in the proof of the main result. The term $(\phi,E)$-chain reflects a remote analogy to $(f,\eps)$-chains in \tl\ dynamics.  Throughout this subsection, we let $A,B$ denote two disjoint subsets of a countable group $G$.

\begin{defn}Given a partially defined bijection $\phi:A'\to B'$, where $A'\subset A$ and $B'\subset B$, such that all multipliers $\phi(a)a^{-1}$ belong to a finite set $E\subset G$, by a \emph{$(\phi,E)$-chain of length $2n$} (or briefly just \emph{a chain}) we will mean a \sq\ $\mathbf C=(a_1,b_1,a_2,b_2,\dots,a_n,b_n)$ of \,$2n$ \emph{different} elements alternately belonging to $A$ and $B$, such that
$$
\text{for each }i=1,2,\dots,n,\ \  b_i\in Ea_i,
$$
and
$$
\text{for each }i=1,2,\dots,n-1, \ \ b_i\in B',\ \ a_{i+1}\in A' \text{ \ and \ } b_i = \phi(a_{i+1})
$$
(in particular, $b_i\in Ea_{i+1}$).
\end{defn}
The $(\phi,E)$-chains starting at a point $a_1\in A\setminus A'$ and ending at a point $b_n \in B\setminus B'$ are of special importance, as they allow one to ``correct'' the mapping and include $a_1$ in the domain and $b_n$ in the range.

\begin{defn}
A $(\phi,E)$-chain $\mathbf C=(a_1,b_1,a_2,b_2,\dots,a_n,b_n)$ will be called a \emph{$\phi$-correction chain} if $a_1\in A\setminus A'$ and $b_n\in B\setminus B'$.
With each $\phi$-correction chain $\mathbf C$ we associate the \emph{correction of $\phi$ along $\mathbf C$}. The corrected map denoted by $\phi^{\mathbf C}$ is defined on $A'\cup\{a_1\}$ onto $B'\cup\{b_{n}\}$, as follows: for each $i=1,2,\dots,n$ we let
$$
\phi^{\mathbf C}(a_i) = b_i,
$$
and for all other points $a\in A'$ we let $\phi^{\mathbf C}(a)=\phi(a)$.
\end{defn}

The correction may be visualized as follows (solid arrows in the top row represent the map $\phi$ and in the bottom row they represent $\phi^\mathbf C$; the dashed arrows represent the ``$E$-proximity relation'' $b\in Ea$):
\begin{gather*}
a_1\dashrightarrow b_1\longleftarrow a_2\dashrightarrow b_2\longleftarrow a_3\ \dots\ b_{n-1}\longleftarrow a_n\dashrightarrow b_n\\
\Downarrow\\
a_1\longrightarrow b_1\dashleftarrow a_2\longrightarrow b_2\dashleftarrow a_3\ \dots\ b_{n-1}\dashleftarrow a_n\longrightarrow b_n
\end{gather*}
(the dashed arrows become solid, the solid arrows are removed from the map). Notice that $\phi^{\mathbf C}$ still has all its multipliers $\phi^{\mathbf C}(a) a^{- 1}$ in the set $E$.
\smallskip

The problem with the correction chains is that the corresponding corrections of $\phi$ usually cannot be applied simultaneously. The correction chains may collide with each other, i.e., pass through common points and then the corresponding corrections rule each other out. To manage this problem we need to learn more about the possible collisions and then carefully select a family of mutually non-colliding correction chains. The details of this selection are given below.

\begin{defn}
Two $\phi$-correction chains \emph{collide} if they have a common point.
\end{defn}

Since the starting points of $\phi$-correction chains belong to $A\setminus A'$, the ending points belong to $B\setminus B'$, other odd points (counting along the chain) belong to $A'$, other even points belong to $B'$, where the above four sets are disjoint, and each even point is tied to the following odd point by the inverse map $\phi^{-1}$, each collision between two $\phi$-correction chains, say
$\mathbf C=(a_1,b_1,a_2,b_2,\dots,a_n,b_n)$ and $\mathbf C'=(a'_1,b'_1,a'_2,b'_2,\dots,a'_m,b'_m)$, is of one of the following three types:
\begin{itemize}
	\item \emph{common start}: $a_1=a'_1$,
	\item \emph{common end}: $b_n=b'_m$,
	\item all other collisions occur in pairs $(b_i,a_{i+1})=(b'_j,a'_{j+1})$ for some $1\le i<n$ and $1\le j<m$.
\end{itemize}
Of course, two chains may have more than one collision. Note that the definition of a $(\phi,E)$-chain eliminates the possibility of ``self-collisions'' in one chain.

\begin{defn}
Given a $(\phi,E)$-chain $\mathbf C=(a_1,b_1,a_2,b_2,a_3,\dots,a_n,b_n)$, the \sq\
$\mathbf n(\mathbf C)=(p_1,q_1,p_2,q_2,\dots,p_{n-1},q_{n-1},p_n)$, where
$p_i=b_ia_i^{-1}$ $(i=1,2,\dots,n)$ and $q_i=b_ia_{i+1}^{-1}$ $(i=1,2,\dots,n-1)$,
will be called the \emph{name} of $\mathbf C$.
\end{defn}
Notice that the name is always a \sq\ of elements of $E$, of length $2n-1$.

\begin{lem}\label{shorter}
If two different $\phi$-correction chains have the same name (note that their lengths are then equal) and collide with each other then each of them collides also with a strictly shorter $\phi$-correction chain.
\end{lem}

\begin{proof}
It is obvious that if two $\phi$-correction chains with the same name, say
$$
\mathbf C=(a_1,b_1,a_2,b_2,a_3,\dots,a_n,b_n), \ \ \mathbf C'=(a'_1,b'_1,a'_2,b'_2,a'_3,\dots,a'_n,b'_n),
$$
have the common start $a_1=a_1'$ or the common end $b_n=b_n'$, or a common pair
$(b_i, a_{i+1})=(b'_i,a'_{i+1})$ with the same index $i=1,2,\dots,n-1$, then the chains are equal. The only possible collision between two different $\phi$-correction chains with the same name is that they have a common pair $(b_i,a_{i+1})=(b'_j,a'_{j+1})$ with $i\neq j$. Let $i_0$ be the smallest index appearing in the role of $i$ or $j$ in the collisions
of $\mathbf C$ with $\mathbf C'$ and assume that it plays the role of $i$ (with some corresponding $j$).  Then
$$
(a_1,b_1,a_2,b_2,a_3,\dots,a_{i_0},b_{i_0},a_{i_0+1},b'_{j+1},a'_{j+2},\dots,a'_n,b'_n)
$$
is a $\phi$-correction chain (it has no self-collisions) of length strictly smaller than $2n$, and clearly it collides with both $\mathbf C$ and $\mathbf C'$.
\end{proof}

We enumerate $E$ (arbitrarily) as $\{g_1,g_2,\dots,g_k\}$. We define
$$
\mathbf N=\bigcup_{n=1}^\infty E^{\times2n-1},
$$
which means the disjoint union of the $(2n-1)$-fold Cartesian products of copies of $E$.
This set can be interpreted as the collection of all ``potential'' names of the correction chains of
any partially defined bijection from $A$ to $B$ with the multipliers in $E$. The enumeration of $E$ induces the following linear order on $\mathbf N$:
$$
\mathbf n<\mathbf n'\ \ \iff\ \ |\mathbf n|<|\mathbf n'| \ \vee \ (\,|\mathbf n|=|\mathbf n'| \ \wedge\ \mathbf n<\mathbf n'\,),
$$
where $|\mathbf n|$ denotes the length of $\mathbf n$ and the last inequality is with respect to the lexicographical order on $E^{\times|\mathbf n|}$.

\begin{defn}
A $\phi$-correction chain $\mathbf C$ is \emph{minimal} if it does not collide with any other $\phi$-correction chain whose name precedes $\mathbf n(\mathbf C)$ in the above defined order
on~$\mathbf N$.
\end{defn}

\begin{lem}\label{minnox}
Minimal $\phi$-correction chains do not collide with each other.
\end{lem}
\begin{proof}
If two $\phi$-correction chains with different names collide, one of them is not minimal. If two $\phi$-correction chains with the same name collide, by Lemma \ref{shorter} none of them is minimal.
\end{proof}

\begin{lem}\label{minch}
Assume that $E$ is a symmetric set containing the unity $e$ and let $a_1\in A\setminus A'$. If there is a $\phi$-correction chain $\mathbf C$ of length $2n$, starting at $a_1$, then there exists a minimal $\phi$-correction chain of length at most $2n$ contained in the finite set $E^{s(n)}a_1$ (where $s(n)$ depends only
on $|E|$ and $n$).
\end{lem}

\begin{proof}
If $\mathbf C$ itself is not minimal then it collides with a $\phi$-correction chain $\mathbf C_1$ with $\mathbf n(\mathbf C_1)<\mathbf n(\mathbf C)$ in $\mathbf N$. Clearly, $\mathbf C_1$ is entirely contained in $E^{4n}a_1$. If $\mathbf C_1$ is not minimal, then it collides with some $\mathbf C_2$, whose name precedes that of $\mathbf C_1$ (and hence also that of $\mathbf C$). Now, $\mathbf C_2$ is contained in $E^{6n}a_1$. This recursion may be repeated at most $\sigma_n-1=\sum_{i=1}^n|E|^{2n}-1$ times, because this number estimates the number of names preceding $\mathbf n(\mathbf C)$. So, before $\sigma_n$ steps are performed, a minimal $\phi$-correction chain must occur. Its length is at most $2n$ and it is entirely contained in $E^{2n\sigma_n}a_1$.
\end{proof}

It is the following lemma, where subexponentiality of the group comes into play.

\begin{lem}\label{key}
Let $G$ be a subexponential group. Let $\CT$ be a tiling of $G$ and let $\mathcal S$ denote the set of all shapes of $\CT$. Denote $E=\bigcup_{S\in\mathcal S}SS^{-1}$. Let $A,B$ be disjoint subsets of $G$ satisfying, for some $\eps>0$ and every tile $T$ of $\CT$, the inequality
$$
|B\cap T|-|A\cap T|>\eps|T|.
$$
Let $N\ge 1$ be such that for any $n\ge N$,
$$
\frac1n\log|(E^2)^n|<\log(1+\eps)
$$
(by the subexponentiality assumption, since $E^2$ is finite, such an $N$ exists).
Then, for any partially defined bijection $\phi:A'\to B'$ with $A'\subset A,\ B'\subset B$, such that all multipliers $\phi(a)a^{-1}$ are in $E$, for every point $a_1\in A\setminus A'$, there exists a $\phi$-correction chain of length at most $2N$, starting at $a_1$ (and ending in $B\setminus B'$).
\end{lem}

\begin{proof}
For each tile $T$ of $\CT$ we have
$$
\frac{|B\cap T|}{|A\cap T|}\ge\frac{\eps|T|}{|A\cap T|}+1\ge 1+\eps
$$
(including the case when the denominator equals $0$).
Clearly, any \emph{$\CT$-saturated} finite set $Q$, i.e, being a union of tiles of $\CT$,
also satisfies
$$
\frac{|B\cap Q|}{|A\cap Q|}\ge1+\eps.
$$
For a set $P\subset G$, we define the \emph{$\CT$-saturation} $P^\CT$ of $P$ as the union of all tiles intersecting $P$:
$$
P^\CT= \bigcup\{T\in\CT:P\cap T\neq\emptyset\}.
$$
Obviously, $P^\CT\subset EP$.

Consider a point $a_1\in A\setminus A'$ (if $A\setminus A'=\emptyset$ then the statement of the theorem holds trivially). Let $T$ be the tile of $\CT$ containing $a_1$, i.e., $T=\{a_1\}^\CT$. Since $T$ contains $a_1$ (and thus $|A\cap T|\ge 1$), we have $|B\cap T|\ge 1+\eps$. There exist $(\phi,E)$-chains of length $2$ from $a_1$ to every $b\in B\cap T$. Now, there are two options:
\begin{itemize}
	\item either at least one of these chains is a $\phi$-correction chain (and then the construction is finished),
	\item or none of these chains is a $\phi$-correction chain, i.e., $B'\cap T=B\cap T$.
\end{itemize}
In the latter option we have $|B'\cap T|=|B\cap T|\ge1+\eps$, i.e., denoting
$$
P_1=\{a_1\} \text{ \ and \ } Q_1=T=P_1^{\CT},
$$
we have
$$
|B'\cap Q_1|\ge1+\eps.
$$

From now on we continue by induction. Suppose that for some $n\ge 1$ we have defined a $\CT$-saturated set $Q_n$ such that
\begin{enumerate}
	\item for every $b\in B\cap Q_n$ there exists a $(\phi,E)$-chain of length at most $2n$ from $a_1$ to $b$,
	\item $B\cap Q_n= B'\cap Q_n$ (i.e., there are no $\phi$-correction chains starting at $a_1$ and ending in
	$Q_n$), and
	\item $|B'\cap Q_n|\ge(1+\eps)^n$.
\end{enumerate}
Then we define $P_{n+1}=\phi^{-1}(Q_n)=\phi^{-1}(B'\cap Q_n)$. Bijectivity of $\phi$ implies that $|P_{n+1}|\ge(1+\eps)^n$. Let $Q_{n+1}$ denote the $\CT$-saturation $P_{n+1}^{\CT}$. Every point $b\in B\cap Q_{n+1}$ is of the form $g\phi^{-1}(b')$ with $g\in E$ and $b'\in B'\cap Q_n$, and, by (1), $b'$ can be reached from $a_1$ by a $(\phi,E)$-chain of length at most $2n$. Thus there exists a $(\phi,E)$-chain of length at most $2(n+1)$ from $a_1$ to every $b\in B\cap Q_{n+1}$. There are two options:
\begin{itemize}
	\item either at least one of these chains is a $\phi$-correction chain (then the construction is finished),
	\item or $B\cap Q_{n+1}=B'\cap Q_{n+1}$.
\end{itemize}
Suppose the latter option occurs. Since $Q_{n+1}$ is $\CT$-saturated, we have
$$
|B'\cap Q_{n+1}|=|B\cap Q_{n+1}|\ge(1+\eps)|A\cap Q_{n+1}|\ge(1+\eps)|P_{n+1}|\ge(1+\eps)^{n+1}.
$$
Now, (1)--(3) are fulfilled for $n+1$, so the induction can be continued.

Notice that for each $n$, $Q_n\subset EP_n$ and, by symmetry of the set $E$, $P_{n+1}\subset EQ_n$. As a consequence, we have $Q_{n+1}\subset E^{2n+1} a_1\subset (E^2)^{n+1} a_1$, and if the latter of the above options occurs, we have
$$
|(E^2)^{n+1}|\ge|Q_{n+1}|\ge|B'\cap Q_{n+1}|\ge(1+\eps)^{n+1},
$$
which implies that $n+1<N$ by the assumption. So, $n=N-2$ is the last integer for which nonexistence of $\phi$-correction chains of length $2(n+1)$ is possible. In the worst case scenario a correcting chain of length $2N$ must already exist.
\end{proof}

\begin{rem}It is absolutely crucial in the proof that we are using a tiling, not a quasitiling leaving some part of $G$ uncovered by the tiles. In such case, $a_1$ may be uncovered by the tiles, moreover, we would have no control as to how many elements of $P_{n+1}=\phi^{-1}(Q_n)$ are ``lost'' in the untiled part of $G$.
\end{rem}

\subsection{The main result}
\begin{thm}\label{main}
Every subexponential group $G$ has the comparison property.
\end{thm}

\begin{proof}
By Proposition \ref{44}, it suffices to prove the theorem for finitely generated groups $G$ with subexponential growth, and Theorem \ref{ujowe} allows us to focus on the Banach density comparison property. So, let $G$ be a finitely generated group with subexponential growth. Let $A,B\subset G$ be disjoint and satisfy $\underline D(B,A)>0$. All we need is, in the subshift $Y^{AB}$, to construct an injection $\tilde\varphi:A\to B$ determined by a block code.

By Lemma \ref{bd}, there exists a finite set $F\subset G$ such that $\underline D_F(B,A)>5\eps$ for some positive $\eps$. By Theorem~\ref{ourtilings}, there exists an $(F,\eps)$-\inv\ tiling $\CT$ of $G$. We let $\mathcal S$ denote the set of all shapes of $\CT$. By Lemma \ref{bdc}, for every shape $S$ of $\CT$ we have $\underline D_S(B,A)>\eps$, in particular,
$$
|B\cap T|-|A\cap T|>\eps|T|,
$$
for every tile $T$ of $\CT$. Let $E=\bigcup_{S\in\mathcal S}SS^{-1}$ and say $E= \{g_1, g_2, \dots, g_k\}$.

We will  build the desired injection $\tilde\varphi:A\to B$ in a series of steps. The first approximation of $\tilde\varphi$ is the map $\phi_1$ defined on a subset of $A$ by
a procedure similar to that used in the proof of Lemma \ref{compisweakcomp}: we let $A_1=A\cap g_1^{-1}(B)$, and $B_1=g_1(A_1)\subset B$ and then, for each $j=2,3,\dots,k$ we define inductively
$$
A_j = A\setminus\Bigl(\bigcup_{i=1}^{j-1}A_i\Bigr)\cap g_j^{-1}\left(B\setminus\Bigl(\bigcup_{i=1}^{j-1} B_i\Bigr)\right)\ \text{and}\ B_j= g_j A_j\subset B.
$$
On each set $A_j$ (with $j=1,2,\dots,k$), $\phi_1$ is defined as the multiplication on the left by $g_j$. We let $A'_1=\bigcup_{i=1}^k A_i\subset A$ and $B'_1=\bigcup_{i=1}^k B_i\subset B$ denote the domain and range of $\phi_1$, respectively. The rule behind the construction of $\phi_1$ is as follows: for each $a\in A$ we first check whether $g_1a\in B$ and for those $a$ for which this is true, we assign $\phi_1(a)=g_1 a$. For other points $a$ we check whether $g_2a\in B$ and, unless $g_2a$ has already been assigned as $\phi_1(a')$ (for some $a'\in A$) in the previous step, we assign $\phi_1(a)=g_2 a$. And so on: in step $i$ we assign $\phi_1(a)=g_i a$ if $g_ia\in B$, unless $g_ia$ has already been assigned as $\phi(a')$ (for some $a'\in A$) in steps $1,2,\dots,i-1$. We stop when $i=k$. From this description it is easy to see that $\phi_1$ is an injection from $A_1'$ into $B_1'\subset B$. In fact, it is also seen that if $a_1,a_2\in A$ and
$$
a_1(y^{AB})|_{E^k} = a_2(y^{AB})|_{E^k},
$$
then either $\phi_1(a_1)a_1^{- 1}=\phi_1(a_2) a_2^{- 1}$ or both values of $\phi_1(a_1)$ and $\phi_1(a_2)$ are undefined. Using the criterion \eqref{clr} (for a one-element family $\A$), we conclude that $\phi_1$ restricted to its domain $A_1'$ is determined by a block code (with the coding horizon $E^k$). We remark, that the block code determines some extension of $\phi_1$ to the whole group, but we do not care about the values of the code outside $A_1'$ and we still treat $\phi_1$ as undefined outside $A_1'$. If $A'_1=A$ (which is rather unlikely in infinite groups), then the proof is finished.

Otherwise we continue the construction involving the correction chains and the associated corrections. By Lemma \ref{key}, for an appropriate $N$, every element $a_1\in A\setminus A_1'$ is the start of a $\phi_1$-correction chain of length at most $2N$. Next, by Lemma \ref{minch}, within $E^{s(N)}a_1$ there is a minimal $\phi_1$-correction chain of length at most $2N$. Finally, by Lemma \ref{minnox}, all minimal $\phi_1$-correction chains of lengths at most $2N$ do not collide with each other. Thus we can perform simultaneous corrections along all $\phi_1$-correction chains of lengths at most $2N$. The corrected map will be denoted by $\phi_2$. For each $a\in A\setminus A'_1$ perhaps we have not yet included $a$ in the domain $A'_2$ of $\phi_2$, but we have included in $A'_2$ at least one new point from $E^{s(N)}a\cap (A\setminus A_1')$. Clearly, $\phi_2$ sends $A'_2$ into $B$ and the multipliers of $\phi_2$ are contained in $E$.

We will now argue why $\phi_2$ is determined by a block code. Notice that given $a\in A$, finding all $\phi_1$-correction chains of lengths bounded by $2N$ starting at or passing through $a$ requires examining the values of $\phi_1$ at most in the set $E^{2N}a$. Then, given such a chain, we can decide whether it is minimal or not by examining all $\phi_1$-correction chains of lengths bounded by $2N$ which collide with it. For this, viewing the values of $\phi_1$ on the set $E^{4N}a$ suffices. Now suppose that $a_1,a_2\in A$ and
$$
a_1(y^{AB})|_{E^{k+4N}}=a_2(y^{AB})|_{E^{k+4N}}.
$$
Since $E^k$ is the coding horizon for $\phi_1$, we have
$$
a_1(\bar\phi_1)|_{E^{4N}}=a_2(\bar\phi_1)|_{E^{4N}},
$$
where $\bar\phi_1$ is defined as the symbolic element over the alphabet $E\cup\{\emptyset\}$ by the rule
$$
(\bar\phi_1)_g=\begin{cases}\phi_1(g)g^{-1}& \text{if }g\in A'_1,\\ \emptyset &\text{otherwise,}\end{cases}
$$
($g\in G$). This implies that $(r_1 a_1, s_1 a_1, r_2 a_1, s_2 a_1, \dots, r_n a_1, s_n a_1)$ is a (minimal) $\phi_1$-correction chain if and only if $(r_1 a_2, s_1 a_2, r_2 a_2, s_2 a_2, \dots, r_n a_2, s_n a_2)$ is a (minimal) $\phi_1$-correction chain, whenever $n\le N$ and all $r_i$ and $s_i$ belong to $E^{2N}$. Hence either both $a_1$ and $a_2$ lie on minimal $\phi_1$-correction chains of length at most $2N$, or both do not. In the latter case, since $a_1(y^{AB})|_{E^{k}}=a_2(y^{AB})|_{E^{k}}$, either $\phi_2(a_1)a_1^{-1}=\phi_1(a_1)a_1^{-1}=\phi_1(a_2)a_2^{-1}=\phi_2(a_2)a_2^{-1}$ or both $\phi_2(a_1)$ and $\phi_2(a_2)$ are undefined. In the former case, the lengths and names of the two minimal $\phi_1$-correction chains are the same, moreover $a_1$ and $a_2$ occupy equal positions in the corresponding chains. This implies that the multipliers $\phi_2(a_1)a_1^{-1}$ and $\phi_2(a_2)a_2^{-1}$ (although different than those for $\phi_1$) will both be defined and equal. So, $\phi_2$ is indeed determined by a block code.

The above process can be now repeated: the next map $\phi_3$ is obtained by performing simultaneous corrections along all minimal $\phi_2$-correction chains of lengths not exceeding $2N$. Again, for every $a\in A\setminus A'_2$, at least one point from each
set $E^{s(N)}a$ is included in the domain $A'_3$ of $\phi_3$ (the intersection $(A\setminus A'_2)\cap E^{s(N)}a$ is nonempty as it contains $a$, and often $a$ will be the new point included in $A'_3$). By the same arguments as before, the map $\phi_3$ is an injection from $A'_3$ into $B$ determined by a block code (with the coding horizon $E^{k+ 4N}$), and the multipliers of $\phi_3$ remain in $E$.

We claim that after a finite number $m$ of analogous steps all points of $A$ will be included in the domain of $\phi_m$, i.e., $\phi_m$ will be the desired injection $\tilde\varphi$ from $A$ into $B$. Indeed, a point $a\in A\setminus A_1'$ remains outside the domains of all the maps $\phi_i$ with $i\le m$ only if the number of all other points (except $a$) in $(A\setminus A_1')\cap E^{s(N)}a$ is at least $m- 1$ (because in each step at least one new point from this set is included in the domain). This is clearly impossible for $m> |E^{s(N)}|$, hence the desired finite number $m$ exists. By induction, all the maps $\phi_i$ ($i=1,2,\dots,m$) are determined by block codes (the coding horizon for the code which determines $\tilde\varphi=\phi_m$ is at most the set $E^{k+4Nm}$). This ends the proof.
\end{proof}

\subsection{Two questions}
As we have already mentioned, the problem whether all countable amenable groups have the comparison property is rather difficult. On the other hand, based on the experience with subexponential groups, one might hope that other additional assumptions might help as well. We formulate two relaxed, yet still open, versions of Question \ref{3.7}.

\begin{ques}
\begin{enumerate}
	\item Do all countable amenable residually finite groups have the comparison property?
	\item Do all countable amenable left (right) orderable groups have the comparison property?
\end{enumerate}
\end{ques}

\section{Free actions and tilings}\label{piec}

In this section we provide an application of comparison to the existence of so-called \emph{dynamical tilings} with good F\o lner properties in free actions on \zd\ compact metric spaces. At the beginning of the paper, we have explained that the existence of such tilings is very important in the study of some areas, for example, in building the theory of symbolic extension for actions of countable amenable groups. Such tilings are guaranteed to exist in $\z$-actions, which follows from various versions of marker theorems (see e.g. \cite{Bo}). But for actions of general countable amenable groups, just like comparison, the existence of dynamical tilings remains an open problem.

\begin{defn}\label{dqt}
Let a countable amenable group $G$ act on a \zd\ compact metric space $X$ and let $\CS$ be a finite family of finite subsets of $G$ (containing the unity $e$). We say that the action \emph{admits a dynamical quastiling with shapes in $\CS$} if there exists a map $x\mapsto \CT_x$, which assigns to every $x\in X$ a quasitiling $\CT_x$ of $G$ with shapes in $\CS$ (see Definition \ref{quasi}), and $x\mapsto \CT_x$ is a factor map from $X$ onto a symbolic dynamical system over the alphabet $\Delta=\CS\cup\{\mathsf 0\}$, where $\CT_x$ is viewed as a point in $\Delta^G$ (see the comments below Definition \ref{quasi}). We say that a dynamical quasitiling is $(K,\eps)$-invariant, $\eps$-disjoint, disjoint, $\alpha$-covering, or that it is a \emph{dynamical tiling} if $\CT_x$ has the respective property for every $x$. We will say that the action has \emph{the tiling property} if, for every finite set $K\subset G$ and every $\eps>0$, it admits a $(K,\eps)$-invariant dynamical tiling.
\end{defn}

The fact that the dynamical quasitiling $x\mapsto\CT_x$ is a \tl\ factor of the action of $G$ on $X$ is equivalent to the conjuction of the following two statements:
\begin{enumerate}
	\item for any finite set $F$ of $G$, if $x$ and $x'$ are sufficiently close to each other in $X$, then the set $F$ is tiled by $\CT_x$ and by $\CT_{x'}$ in the same way,
	\item for each $g\in G$ we have $\CT_{g(x)}=\{Tg^{-1}:T\in\CT_x\}$.
\end{enumerate}

In \cite{DH} the following result is proved:
\begin{thm}{\cite[Corollary 3.5]{DH}}\label{quasitilings}
Let a countable amenable group $G$ act freely on a \zd\ compact metric space $X$. For any finite set $K\subset G$ and any $\eps>0,\ \delta>0$ the action admits a $(K,\eps)$-invariant, disjoint, $(1\!-\!\delta)$-covering dynamical quasitiling $x\mapsto\CT_x$.
\end{thm}

We will now demonstrate strong connection between comparison and the tiling property of actions.

\begin{thm}\label{dt}
Let a countable amenable group $G$ act freely on a \zd\ compact metric space $X$. Then the action admits comparison if and only if it has the tiling property. The backward implication holds without assuming that the action is free.
\end{thm}

\begin{proof}
We need to consider only infinite groups $G$. Firstly we will show that for any finite $K\subset G$ and $1>\eps>0$, the free action admits a $(K, \eps)$-invariant dynamical tiling. By Theorem \ref{quasitilings}, the free action admits a $(K,\frac\eps2)$-invariant, disjoint, $(1\!-\!\delta)$-covering dynamical quasitiling $x\mapsto \CT'_x$, where $\delta> 0$ is so small that $\frac{2\delta}{1-\delta}<\frac\eps{2|K|}$. We denote by $\CS'$ the collection of all shapes used by this quasitiling. We can assume that each shape $S\in\CS'$ has cardinality so large that the interval $(\frac{2\delta}{1-\delta}|S|,\frac\eps{2|K|}|S|)$ contains an integer $i_S$ (if this fails, we can choose a $(K',\frac\eps2)$-invariant, disjoint, $(1\!-\!\delta)$-covering dynamical quasitiling, where $K'\supset K$, and clearly this quasitiling is also $(K,\frac\eps2)$-invariant, while its shapes have cardinalities at least $\frac{|K'|}2$, as large as we wish; here we use infiniteness of $G$). In each shape $S\in\mathcal S'$ we select (arbitrarily) a subset $B_S$ of cardinality $i_S$. Given $x\in X$, we now observe two subsets of $G$:
$$
A_x= G\setminus \bigcup\CT'_x\text{ \ \ and  \ \ }B_x = \bigcup_{(S,c)\in\CT'_x}B_Sc.
$$
Clearly, $\overline D(A_x)=1-\underline D(\bigcup\CT'_x)\le\delta$. Using Lemma~\ref{1.5} we easily get $\underline D(B_x)>(1-\delta)\cdot \frac{2\delta}{1-\delta} =2\delta$. By Corollary \ref{coro}, $\underline D(B_x,A_x)>\delta$. Define two subsets of $X$:
$$
\mathsf A=\{x:e\in A_x\}\text{ \ \ and \ \ }\mathsf B=\{x:e\in B_x\}.
$$
By continuity of the assignment $x\mapsto \CT'_x$, and since one can determine whether $e\in A_x$ (and likewise, whether $e\in B_x$) from the symbolic representation of $\CT'_x$ (which is a subshift over the alphabet $\Delta'=\mathcal S'\cup\{0\}$) by viewing the symbols in a bounded horizon $\bigcup_{S\in\CS'}S^{- 1}$ (independent of $x$) around $e$, both sets $\mathsf A$ and $\mathsf B$ are clopen (and obviously disjoint). The notation $A_x,\,B_x$ is now consistent with \eqref{kiki} and \eqref{kiko} for the sets $\mathsf A,\,\mathsf B$, respectively, hence, by Proposition \ref{prop} (1) (the last equality) we obtain $\underline D(\mathsf B,\mathsf A)\ge\delta>0$. The comparison property of the action on $X$ implies that $\mathsf A\preccurlyeq\mathsf B$.

Since we prefer to work with a symbolic system in place of the \zd\ system $X$, we will now build a symbolic factor $\hat X$ of $X$ carrying the minimum information needed to restore both the dynamical quasitiling $x\mapsto \CT_x'$ and the subequivalence $\mathsf A\preccurlyeq\mathsf B$. Let $\{\mathsf A_1,\mathsf A_2,\dots,\mathsf A_k\}$ and
$g_1,g_2,\dots, g_k$ be, respectively, the clopen partition of $\mathsf A$ and the associated elements of $G$ as in the definition of subequivalence. We define a factor map $\pi:X\to\hat X\subset{\hat\Delta}^G$, where
$\hat\Delta=\Delta'\times\{\mathsf 0,\mathsf 1,\dots,\mathsf k,\mathsf{k+1}\}$,
as follows:
$$
(\pi(x))_g =
\begin{cases}
((\CT_x')_g,\mathsf i)&\text{ \ \ if \ }g(x)\in \mathsf A_i, \ \ i=1,2,\dots,k\\
((\CT_x')_g,\mathsf{k+1})&\text{ \ \ if \ }g(x)\in\mathsf B\\
((\CT_x')_g,\mathsf{0})&\text{ \ \ if \ }g(x)\notin\mathsf A\cup\mathsf B.\\
\end{cases}
$$
Denote by $\hat{\mathsf A}_i=[\cdot,\mathsf i]$, $\hat{\mathsf A}=\bigcup_{i=1}^k[\cdot,\mathsf i]$ and $\hat{\mathsf B}=[\cdot,\mathsf{k+ 1}]$. Clearly, $\pi^{-1}(\hat{\mathsf A})=\mathsf A$, $\pi^{-1}(\hat{\mathsf A}_i)=\mathsf A_i$ ($i=1,2,\dots,k$) and $\pi^{-1}(\hat{\mathsf B})=\mathsf B$, which easily implies that $\hat{\mathsf A}\preccurlyeq\hat{\mathsf B}$ in the subshift $\hat X$, and the subequivalence involves the same elements $g_1,g_2,\dots, g_k$. Also for any $\hat x\in \hat X$ all quasitilings $\CT'_x$ with $x\in\pi^{-1}(\hat x)$ coincide. Hence, the subshift $\hat X$ admits a dynamical quasitiling $\hat x\mapsto\CT'_{\hat x}$, where $\CT'_{\hat x}=\CT'_x$ for any $x\in\pi^{-1}(\hat x)$.

By Theorem \ref{tutka} (1) (and its proof), there exists a family of injections $\tilde\varphi_{\hat x}:\hat A_{\hat x}\to \hat B_{\hat x}$ indexed by $\hat x\in \hat X$ (according to our convention, $\hat A_{\hat x}=\{g:g(\hat x)\in\hat{\mathsf A}\}$, $\hat B_{\hat x}=\{g:g(\hat x)\in\hat{\mathsf B}\}$), determined by a block code $\Xi:\hat\Delta^F\to E$,
where $E=\{g_1,g_2,\dots, g_k\}$ (and $F$ is a finite coding horizon). As easily verified, if $\hat x=\pi(x)$ then $\hat A_{\hat x} = A_x$ and $\hat B_{\hat x}=B_x$, thus, for any $x\in X$ we can define injections $\tilde\varphi_x:A_x\to B_x$ simply as $\tilde\varphi_{\hat x}$.

Now we are in a position to modify the quasitilings $\CT'_x$. Given $x\in X$, we define a transformation of the tiles $Sc\in\CT'_x$ as follows:
$$
\Phi_x(Sc)=Sc\cup\tilde\varphi_x^{-1}(B_Sc)\subset Sc\cup A_x
$$
(recall that $B_Sc$ is a part of the set $B_x$, so its preimage by $\tilde\varphi_x$ is a part of $A_x$). We will call the set $\tilde\varphi_x^{-1}(B_Sc)$ the \emph{added set}.
We define the center of the new tile $\Phi_x(Sc)$ as $c$. The shape of the new tile equals
$$
\Phi_x(Sc)c^{-1}=S\cup\tilde\varphi_x^{-1}(B_Sc)c^{-1}.
$$
Note that
$$
\tilde\varphi_x^{-1}(B_Sc)c^{-1}\subset E^{-1} (B_Sc)c^{-1}\subset E^{-1}S,
$$
which is a finite set. Since $\mathcal S'$ is finite, the set $\mathcal S$ of all new shapes is also finite. As the quasitiling $\CT'_x$ is disjoint, $\tilde\varphi_x$ restricted to $A_x$ is injective, and the image of $A_x$ is contained in $B_x=\bigcup_{Sc\in\CT'_x}B_Sc$, it is clear that the new quasitiling
$$
\CT_x=\{\Phi_x(Sc):Sc\in\CT'_x\}
$$
is a tiling (disjoint and covering $G$ completely).

Further, for any tile $Sc$ of $\CT'_x$ the added set $\tilde\varphi_x^{-1}(B_Sc)$
has cardinality at most $|B_S|=i_S<\frac{\eps}{2|K|}|S|$. Thus
$$
|K\Phi_x(Sc)|\le |KSc| + |K|\cdot \frac{\eps}{2|K|}|S| = |KS|+\frac\eps2|S|.
$$
We can assume (at the beginning of the proof) that $e\in K$, and then $(K,\frac\eps2)$-invariance of $S$ is equivalent to the inequality $|KS|<(1+\frac\eps2)|S|$. Thus
$$
|K\Phi_x(Sc)|< (1+\eps)|S|\le (1+\eps)|\Phi_x(Sc)|,
$$
and so $\Phi_x(Sc)$ is $(K,\eps)$-invariant. Summarizing, we have constructed a mapping $x\mapsto\CT_x$ into tilings with a finite set $\mathcal S$ of $(K,\eps)$-invariant shapes.

\smallskip
We need to show that the assignment $x\mapsto\CT_x$ is a dynamical tiling, i.e., a \tl\ factor map from $X$ to a subshift over the alphabet $\Delta = \mathcal S\cup\{0\}$. Of course, it suffices to show that $x\mapsto\CT_x$ ``factors through'' $\hat X$, i.e., that
$\CT_x$ depends in fact on $\hat x=\pi(x)$ and the dependence is via a block code.
To do so, we can use the criterion \eqref{clr}, i.e., we need to indicate a finite set $J\subset G$, such that for any $x_1,x_2\in X$ and $g\in G$,
\begin{equation}\label{equ}
\hat x_1|_{Jg}=\hat x_2|_{Jg} \implies (\CT_{x_1})_g=(\CT_{x_2})_g,
\end{equation}
where $\hat x_1=\pi(x_1)$ and $\hat x_2=\pi(x_2)$.

We claim that the set $J=\{e\}\cup FE^{-1}R$ is good, where $F$ is the finite coding horizon of $\Xi$ and $R=\bigcup_{S\in\CS'}S$. In order to verify this claim assume that with so defined $J$ the left hand side of \eqref{equ} holds for some $x_1,x_2\in X$ and $g\in G$. Since $g\in Jg$, and the first entries of the pairs which constitute the symbols $(\hat x_1)_g$ and $(\hat x_2)_g$ equal $(\CT'_{x_1})_g$ and $(\CT'_{x_2})_g$, respectively, we have $(\CT'_{x_1})_g=(\CT'_{x_2})_g$. If this common entry is $0$ then $g$ is not a center of any tile in neither $\CT'_x$ nor $\CT'_{x'}$, and then $g$ is not a center of any tile in neither $\CT_{x_1}$ nor $\CT_{x_2}$, i.e., $(\CT_{x_1})_g=(\CT_{x_2})_g=0$. If the common entry is some $S\in\mathcal S'$ then we know that $g=c$ is a center of some tile in both $\CT_{x_1}$ and $\CT_{x_2}$, moreover the shapes of these tiles have the common part $S$ and may differ only in having different added sets. The added sets equal $\tilde\varphi_{x_1}^{-1}(B_Sc)c^{-1}$ and $\tilde\varphi_{x_2}^{-1}(B_Sc)c^{-1}$, respectively.
Since we can replace the subscripts $x_1,x_2$ correspondingly by $\hat x_1,\hat x_2$, we just need to show that
$$
\tilde\varphi_{\hat x_1}^{-1}(B_Sc)=\tilde\varphi_{\hat x_2}^{-1}(B_Sc).
$$
Since $FE^{-1}Rc\subset Jg$, the left hand side of \eqref{equ} implies $\hat x_1|_{FE^{-1}Rc}=\hat x_2|_{FE^{-1}Rc}$. Recall that the family $\{\tilde\varphi_{\hat x}\}_{\hat x\in\hat X}$ is determined by a block code with coding horizon $F$. We deduce that
$\tilde\varphi_{\hat x_1}$ agrees with $\tilde\varphi_{\hat x_2}$ on the set $E^{-1}Rc$, which contains $E^{-1}Sc$, which contains $E^{-1}B_Sc$. But $E^{-1}B_Sc$ contains the union
$\tilde\varphi_{\hat x_1}^{-1}(B_Sc)\cup \tilde\varphi_{\hat x_2}^{-1}(B_Sc)$. Since
$\tilde\varphi_{\hat x_1}$ and $\tilde\varphi_{\hat x_2}$ agree on this union, we conclude that $\tilde\varphi_{\hat x_1}^{-1}(B_Sc)=\tilde\varphi_{\hat x_2}^{-1}(B_Sc)$, which ends the proof of the first implication.
\medskip

Now we shall show how dynamical tilings can be used to prove comparison. Assume that $G$ acts on a \zd\ compact metric space $X$ (we do not assume freeness of the action) and that for every finite $K\subset G$ and $\eps>0$ this action admits a dynamical tiling with $(K,\eps)$-invariant shapes. Let $\mathsf A,\mathsf B$ be disjoint clopen subsets of $X$ such that $\mu(\mathsf B)>\mu(\mathsf A)$ for all \im s $\mu$ on $X$. We need to show that $\mathsf A\preccurlyeq\mathsf B$.

As we have observed in Remark \ref{from0}, the infimum $\inf_{\mu\in\M_G(X)}(\mu(\mathsf B)-\mu(\mathsf A))$ is positive. Proposition \ref{prop} (1) implies that
$$
\underline D(\mathsf B,\mathsf A)\ge6\eps,
$$
for some $\eps>0$. By Lemma \ref{bbb}, there exists a finite set $F\subset G$ satisfying, for every $x\in X$, $\underline D_F(B_x,A_x)\ge5\eps$. By the tiling property, there exists a dynamical tiling $x\mapsto\CT_x$ with some set of shapes $\mathcal S$ such that each shape $S\in\mathcal S$ is $(F,\eps)$-\inv. Lemma~\ref{bdc} implies that for every $S\in\mathcal S$ and $x\in X$, we have
$$
\underline D_S(B_x,A_x)\ge \underline D_F(B_x,A_x)-4\eps>0,
$$
which yields $|A_xg^{-1}\cap S|<|B_xg^{-1}\cap S|$ for every $g\in G$.

Similarly, as in the preceding proof, we will build a symbolic factor $\hat X$ of $X$ carrying the minimum information about both the sets $\mathsf A,\mathsf B$ and the dynamical tiling. Namely, we define a factor map $\pi:X\to \hat X\subset {\hat\Delta}^G$, where this time $\hat\Delta = \{\mathsf 0,\mathsf 1,\mathsf 2\}\times\Delta$ (as usually, $\Delta=\mathcal S\cup\{0\}$ is the alphabet in the symbolic representation of the dynamical tiling), as follows
$$
(\pi(x))_g=
\begin{cases}
(\mathsf1,S)& \ \text{ if \ \ }g\in A_x, Sg\in\CT_x\\
(\mathsf2,S)& \ \text{ if \ \ }g\in B_x, Sg\in\CT_x\\
(\mathsf0,S)& \ \text{ if \ \ }g\notin A_x\cup B_x, Sg\in\CT_x\\
(\mathsf1,0)& \ \text{ if \ \ }g\in A_x, Sg\notin\CT_x\\
(\mathsf2,0)& \ \text{ if \ \ }g\in B_x, Sg\notin\CT_x\\
(\mathsf0,0)& \ \text{ if \ \ }g\notin A_x\cup B_x, Sg\notin\CT_x.
\end{cases}
$$
As before, the subshift $\hat X$ admits a dynamical tiling $\hat x\mapsto\CT_{\hat x}$, where $\CT_{\hat x}=\CT_x$ for any $x\in\pi^{-1}(\hat x)$.
Denote $\hat{\mathsf A} = [\mathsf 1,\cdot]$ and $\hat{\mathsf B} = [\mathsf 2,\cdot]$.
We have $\mathsf A=\pi^{-1}(\hat{\mathsf A})$ and $\mathsf B=\pi^{-1}(\hat{\mathsf B})$.

Thus it suffices to show that $\hat{\mathsf A}\preccurlyeq\hat{\mathsf B}$ in $\hat X$.
By Theorem \ref{tutka} (1), the proof will be ended once we will have constructed a family of injections $\tilde\varphi_{\hat x}:\hat A_{\hat x}\to \hat B_{\hat x}$ indexed by $\hat x\in\hat X$ and determined by a block code.

By the definition of $\pi$ we have, that if $\hat x=\pi(x)$ then $A_x=\hat A_{\hat x}$ and $B_x=\hat B_{\hat x}$, and the inequality $|A_xg^{-1}\cap S|<|B_xg^{-1}\cap S|$ translates to $|\hat A_{\hat x}g^{-1}\cap S|<|\hat B_{\hat x}g^{-1}\cap S|$ (for each $\hat x\in \hat X$, $S\in\mathcal S$ and $g\in G$). In other words, in every block $g(\hat x)|_S$ there are more symbols $\mathsf 2$ than $\mathsf 1$ (we just consider the first entries in the pairs which constitute the symbols). Since $\mathcal S$ is finite and for each $S\in\mathcal S$ there are only finitely many blocks $\mathbf B\in{\hat\Delta}^S$, we have globally a finite number of possible blocks $\mathbf B$ appearing in the role $g(\hat x)|_S$ (with $\hat x\in \hat X$, $g\in G$ and $S\in\mathcal S$). For every block $\mathbf B$ in this finite collection we select arbitrarily an injection $\varphi_{\mathbf B}:\{s\in S:\mathbf B(s)=(\mathsf 1,\cdot)\}\to\{s\in S:\mathbf B(s)=(\mathsf 2,\cdot)\}$, where $S$ is the domain of $\mathbf B$.

Fix some $\hat x\in \hat X$ and $a\in\hat A_{\hat x}$. Let $Sc$ be the tile of $\CT_{\hat x}$ containing $a$ and let $\mathbf B = c(\hat x)|_S$. We define
$$
\tilde\varphi_{\hat x}(a) = \varphi_{\mathbf B}(ac^{-1})c.
$$
Since $\mathbf B(ac^{-1})=\hat x_a =(\mathsf 1,\cdot)$, $\varphi_{\mathbf B}(ac^{-1})$ is defined and satisfies $\mathbf B (\varphi_{\mathbf B}(ac^{-1}))=(\mathsf 2,\cdot)$, and thus
$\hat x_{\varphi_{\mathbf B}(ac^{-1})c}=(\mathsf 2,\cdot)$, i.e., $\tilde\varphi_{\hat x}(a)\in\hat B_{\hat x}$. Notice that $\tilde\varphi_{\hat x}(a)$ belongs to the same tile of $\CT_{\hat x}$ as $a$. Injectivity of so defined $\tilde\varphi_{\hat x}$ is very easy. Consider $a_1\neq a_2\in \hat A_{\hat x}$. If both elements belong to the same tile of $\CT_{\hat x}$, then their images are different by injectivity of $\varphi_{\mathbf B}$, where $\mathbf B= c(\hat x)|_S$. If they belong to different tiles, their images also belong to different tiles, hence are different. The last thing to check is the condition \eqref{clr}, which will establish that the family $\{\tilde\varphi_{\hat x}\}_{\hat x\in\hat X}$ is determined by a block code. We claim that the horizon
$E=\bigcup_{S\in\mathcal S}SS^{-1}$ is good. Indeed, suppose, for some $\hat x_1,\hat x_2\in \hat X$ and $a_1\in\hat A_{\hat x_1},a_2\in\hat A_{\hat x_2}$, that
\begin{equation} \label{add}
a_1(\hat x_1)|_E = a_2(\hat x_2)|_E.
\end{equation}
Let $S_1c$ be the (unique) tile of $\CT_{a_1(\hat x_1)}$ containing the unity $e$. Then the second entry of the pair constituting the symbol $(a_1 (\hat x_1))_c$ equals $S_1$.
Since $c\in\bigcup_{S\in\mathcal S}S^{-1}\subset E$, by \eqref{add} we obtain that the second entry of the symbol $(a_2 (\hat x_2))_c$ also equals $S_1$, so that $S_1c$ is the (unique) tile of $\CT_{a_2(\hat x_2)}$ containing $e$. Further, since $S_1c\subset E$, by \eqref{add} we have $a_1(\hat x_1)|_{S_1c}= a_2(\hat x_2)|_{S_1c}$ and hence $ca_1(\hat x_1)|_{S_1}= ca_2(\hat x_2)|_{S_1}$. That is, these two restrictions define the same block $\mathbf B\in {\hat\Delta}^{S_1}$. This implies that both $\tilde\varphi_{\hat x_1}(a_1)$ and $\tilde\varphi_{\hat x_2}(a_2)$ are defined with the help of the same injection
$\varphi_\mathbf B$, and
$$
\tilde\varphi_{\hat x_1}(a_1) = \varphi_{\mathbf B}(a_1c_1^{-1})c_1, \ \ \ \ \tilde\varphi_{\hat x_2}(a_2) = \varphi_{\mathbf B}(a_2c_2^{-1})c_2,
$$
where $c_1$ is the center of the tile of $\CT_{\hat x_1}$ containing $a_1$ and $c_2$
is the center of the tile of $\CT_{\hat x_2}$ containing $a_2$. By shift equivariance of the dynamical tiling, we easily see that $c_1=ca_1$ and $c_2=ca_2$, which yields
$$
\tilde\varphi_{\hat x_1}(a_1)a_1^{-1} = \varphi_{\mathbf B}(c^{-1})c=\tilde\varphi_{\hat x_2}(a_2)a_2^{-1}.
$$
This is exactly the condition $\eqref{clr}$ and the proof is finished.
\end{proof}

Combining Theorem \ref{main} with Theorem \ref{dt} we obtain:

\begin{cor}
If $G$ is a subexponential group then every action of $G$ on a \zd\ compact metric space has the tiling property.
\end{cor}

We conclude the paper with a question. Let us say that a countable amenable group $G$ has the \emph{tiling property} if any free action of $G$ on a \zd\ compact metric space has the tiling property as in Definition \ref{dqt}. In such case, by Theorem~\ref{dt} any free action on a \zd\ compact metric space admits comparison. It is easy to see that the tiling property cannot be extended (without modifying the definition) to non-free actions. However, there are \emph{a priori} no obvious reasons why the comparison property could not be extended. Thus the following question is very natural:

\begin{ques}
Is it true that if $G$ has the tiling property (which depends on free actions only) then it also has the comparison property (which depends on all actions; of course we restrict our attention to \zd\ compact metric spaces).
\end{ques}

\end{document}